\documentclass[aap,preprint]{imsart}

\RequirePackage{amsthm,amsmath,amsfonts,amssymb}
\RequirePackage[numbers]{natbib}
\RequirePackage[colorlinks,citecolor=blue,urlcolor=blue]{hyperref}
\RequirePackage{graphicx}

\usepackage{appendix}

\startlocaldefs

\newtheorem{theorem}{Theorem}[section]
\newtheorem{lemma}[theorem]{Lemma}
\theoremstyle{remark}
\newtheorem{definition}[theorem]{Definition}

\newtheorem{corollary}[theorem]{Corollary}

\newtheorem{assumption}{Assumption}

\newtheorem{remark}[theorem]{Remark}

\numberwithin{equation}{section}

\newcommand{\E}{{\mathbb E}}
\newcommand{\R}{{\mathbb R}}

\newcommand{\BMO}{L^{2, \;\mathrm{BMO}}_{\mathcal F^{W}}(0, T;\mathbb{R}^{n})}
\usepackage{url}
\usepackage{hyperref}
\newcommand{\esssup}{\ensuremath{\operatorname{ess\;sup}}}
\newcommand{\cH}{\ensuremath{\mathcal{H}}}
\newcommand{\cM}{\ensuremath{\mathcal{M}}}
\newcommand{\argmin}{\ensuremath{\operatorname*{argmin}}}

\newcommand{\dualgamma}{\widehat{\Gamma}}
\endlocaldefs

\begin{document}

\begin{frontmatter}
\title{Constrained stochastic LQ control with regime switching and application to portfolio selection}
\runtitle{LQ control with regime switching}

\begin{aug}
\author[A]{\fnms{YING} \snm{HU}\ead[label=e1]{ying.hu@univ-rennes1.fr}},
\author[B]{\fnms{XIAOMIN} \snm{SHI}\ead[label=e2]{shixm@mail.sdu.edu.cn}}
\and
\author[C]{\fnms{ZUO QUAN} \snm{XU}\ead[label=e3]{maxu@polyu.edu.hk}}
\address[A]{Univ Rennes, CNRS, IRMAR-UMR 6625, F-35000 Rennes, France,
\printead{e1}}

\address[B]{School of Mathematics and
Quantitative Economics,
Shandong University of Finance and Economics, Jinan, Shandong, China,
\printead{e2}}

\address[C]{Department of Applied Mathematics, The Hong Kong Polytechnic University, Hong Kong, China,
\printead{e3}}
\end{aug}

\begin{abstract}
This paper is concerned with a stochastic linear-quadratic optimal control problem with regime switching, random coefficients, and cone control constraint. The randomness of the coefficients comes from two aspects: the Brownian motion and the Markov chain. Using It\^{o}'s lemma for Markov chain, we obtain the optimal state feedback control and optimal cost value explicitly via two new systems of extended stochastic Riccati equations (ESREs). We prove the existence and uniqueness of the two ESREs using tools including multidimensional comparison theorem, truncation function technique, log transformation and the John-Nirenberg inequality. These results are then applied to study mean-variance portfolio selection problems with and without short-selling prohibition with random parameters depending on both the Brownian motion and the Markov chain. Finally, the efficient portfolios and efficient frontiers are presented in closed forms.
\end{abstract}



\end{frontmatter}

\section{Introduction}
Linear-quadratic (LQ) optimal control is one of the most important problems in control theory. On one hand, it admits elegant optimal state feedback control and optimal cost value through the famous Riccati equation. On the other hand, it has widely applications in many fields, such as engineering, management science and mathematical finance.

Since the pioneering work of Wonham \cite{Wo}, stochastic LQ problem has been extensively studied by numerous researchers with deterministic and stochastic coefficients. For instance, Bismut \cite{Bi} was the first one that studied stochastic LQ problems with random coefficients. Kohlmann and Zhou \cite{KZ} established the relationship between stochastic LQ problems and backward stochastic differential equations. Chen, Li and Zhou \cite{CLZ} studied the indefinite stochastic LQ problem which is different obviously from its deterministic counterpart. Li and Zhou \cite{LxZ} and Li, Zhou and Rami \cite{LZR} studied stochastic LQ problem with Markovian jumps in finite and infinite time horizon respectively. Please refer to Chapter 6 in Yong and Zhou \cite{YZ} for a systematic accounts on this subject.

The stochastic LQ control theory happens to be a powerful tool for solving continuous-time mean-variance portfolio selection problems; see, e.g., \cite{KT, LN, LZL, LZ, Yu, ZL, ZY}. Especially, Li, Zhou and Lim \cite{LZL} studied a mean-variance model with short selling prohibition. Because all the coefficients are assumed to be deterministic, they adopted the Hamilton-Jacobi-Bellman equation and viscosity solution theory. Hu and Zhou \cite{HZ} solved the corresponding problem with random market parameters using stochastic LQ theory combined with the Tanaka's formula. Czichowsky and Schweizer \cite{CS} studied a cone-constrained mean-variance problem in a general semimartingale model.

As is well known, the closed form representation of the optimal control for stochastic LQ control problems relates intimately to the solvability of the corresponding stochastic Riccati equation (SRE). Therefore, the SRE plays a crucial role in studying stochastic LQ problems. It is Kohlmann and Tang \cite{KT}, for the first time, that established
the existence and uniqueness of the one-dimensional SRE. The matrix-valued SRE with uniformly definite coefficients were solved by Tang \cite{Ta}. As for the matrix-valued indefinite SRE, there were only partial results so far; see, e.g., \cite{Du, HZ, QZ}.

In this paper, we study a stochastic LQ control problem with regime switching and random coefficients, where the control variable has to be constrained in a cone. The randomness comes from two aspects: the Brownian motion driving the asset price dynamics and the Markov chain standing for the regime switching. Moreover, the control weighting matrix in the cost functional is allowed to be possibly singular. By the technique of completing squares, we obtain two systems of backward stochastic differential equations (BSDEs) termed extended stochastic Riccati equations (ESREs). These two systems are highly nonlinear, so the solvability of them is interesting in its own right. Thanks to a stability result of BSDE by Cvitanic and Zhang \cite{CZ} and a multidimensional comparison theorem by Hu and Peng \cite{HP}, we could prove the existence of solutions to the two ESREs. To prove uniqueness, most of the aforementioned papers used the Feynman-Kac type representation of SREs. Rather than such an indirect method, in this paper we provide a direct approach using $\log$ transformation and the John-Nirenberg inequality. Finally, we succeed in obtaining the optimal state feedback control and optimal cost value similar to the classical unconstrained-control problem or the problem without regime switching by the two systems of ESREs.

Another economic motivation of this paper is to study continuous-time mean-variance portfolio selection problems with more realistic assumptions that can better reflect random market environment. A Markov chain is usually adopted to reflect the market status in the literature.
For instance, Zhou and Yin \cite{ZY} considered a mean-variance portfolio selection with regime switching, in which the coefficients depended on the market status but not on the Brownian motion.
In practice, however, the market parameters, such as the interest rate, stock appreciation rates and volatilities are affected by the uncertainties caused by the Brownian motion. Thus, it is too restrictive to set market parameters as constants even if the market status is known. From practical point of view, it is necessary to allow the market parameters to depend on both the Brownian motion and the Markov chain.
This paper aims to generalise Zhou and Yin's \cite{ZY} model to a constrained one, in which
the coefficients depend on both the Brownian motion and the Markov chain. We first introduce a system of risk adjust processes $H(i)$, which solves a multidimensional linear BSDEs with unbounded coefficients. We establish the existence and uniqueness of the linear system by contraction mapping method.
To the end, we solve the portfolio selection problem explicitly and completely using the results of the stochastic LQ problem that has been solved.

This paper is organised as follows. In Section 2, we formulate a stochastic LQ problem with regime switching, random coefficients, and portfolio constraint. Section 3 is concerned about the global solvability of two systems of extended stochastic Riccati equations, including existence and uniqueness for the standard and the singular cases. Section 4 gives the solution of the constrained LQ problem. In Section 5, we apply the general results to solve two mean-variance portfolio selection problems with regime switching and with/without portfolio constraints completely.
Finally, Section 6 concludes the paper.


\section{Problem formulation}

Let $(\Omega, \mathcal F, \mathbb{P})$ be a fixed complete probability space on which are defined a standard $n$-dimensional Brownian motion $W(t)=(W_1(t), \ldots, W_n(t))'$ and a continuous-time stationary Markov chain $\alpha_t$ valued in a finite state space $\mathcal M=\{1, 2, \ldots, \ell\}$ with $\ell>1$. We assume $W(t)$ and $\alpha_t$ are independent processes. The Markov chain has a generator $Q=(q_{ij})_{\ell\times \ell}$ with $q_{ij}\geq 0$ for $i\neq j$ and $\sum_{j=1}^{\ell}q_{ij}=0$ for every $i\in\mathcal{M}$.
Define the filtrations $\mathcal F_t=\sigma\{W(s), \alpha_s: 0\leq s\leq t\}\bigvee\mathcal{N}$ and $\mathcal F^W_t=\sigma\{W(s): 0\leq s\leq t\}\bigvee\mathcal{N}$, where $\mathcal{N}$ is the totality of all the $\mathbb{P}$-null sets of $\mathcal{F}$.

\subsection*{Notation}
We use the following notation throughout the paper:
\begin{align*}
L^{2}_{\mathcal{F}}(\Omega;\mathbb{R})&=\Big\{\xi:\Omega\rightarrow
\mathbb{R}\;\Big|\;\xi\mbox { is }\mathcal{F}_{T}\mbox{-measurable, and }\E\big(|\xi|^{2}\big)%
<\infty\Big\}, \\
L^{\infty}_{\mathcal{F}}(\Omega;\mathbb{R})&=\Big\{\xi:\Omega\rightarrow
\mathbb{R}\;\Big|\;\xi\mbox { is }\mathcal{F}_{T}\mbox{-measurable, and essentially bounded}\Big\}, \\
L^{2}_{\mathcal F}(0, T;\mathbb{R})&=\Big\{\phi:[0, T]\times\Omega\rightarrow
\mathbb{R}\;\Big|\;\phi(\cdot)\mbox{ is an }\{\mathcal{F}%
_{t}\}_{t\geq0}\mbox{-adapted process with }\\
&\qquad\mbox{ \ \ \ \ the norm }||\phi||=\Big(\E\int_{0}^{T}|\phi(t)|^{2}dt\Big)^{\frac{1}{2}}<\infty
\Big\}, \\
L^2_{\mathcal{F}}(\Omega;C(0, T;\mathbb{R}))&=\Big\{\phi:[0, T]\times\Omega\rightarrow
\mathbb{R}\;\Big|\;\phi(\cdot)\mbox{ is an }\{\mathcal{F}%
_{t}\}_{t\geq0}\mbox{-adapted process, and}\\
&\qquad\mbox{ \ \ \ \ has continuous sample paths with }\E\Big(\sup_{t\in[0, T]}|\phi(t)|^{2}\Big)<\infty
\Big\}, \\
L^{2, \;\mathrm{loc}}_{\mathcal F}(0, T;\mathbb{R})&=\Big\{\phi:[0, T]\times\Omega\rightarrow
\mathbb{R}\;\Big|\;\phi(\cdot)\mbox{ is an }\{\mathcal{F}%
_{t}\}_{t\geq0}\mbox{-adapted process}\\
&\qquad\mbox{ \ \ \ \ with }\int_{0}^{T}|\phi(t)|^{2}dt<\infty \mbox{ almost surely (a.s.)}
\Big\}, \\
L^{\infty}_{\mathcal{F}}(0, T;\mathbb{R})&=\Big\{\phi:[0, T]\times\Omega
\rightarrow\mathbb{R}\;\Big|\;\phi(\cdot)\mbox{ is an }\{\mathcal{F}%
_{t}\}_{t\geq0}\mbox{-adapted essentially}\\
&\qquad\mbox{ \ \ \ \ bounded process} \Big\}, \\
L^{\infty}_{\mathcal{F}}(\Omega; C(0, T;\mathbb{R}))&=\Big\{\phi:[0, T]\times\Omega
\rightarrow\mathbb{R}\;\Big|\;\phi(\cdot)\mbox{ is an }\{\mathcal{F}%
_{t}\}_{t\geq0}\mbox{-adapted essentially }\\
&\qquad\mbox{ \ \ \ \ bounded process with continuous sample paths} \Big\}.
\end{align*}
These definitions are generalized in the obvious way to the cases that $\mathcal{F}$ is replaced by $\mathcal{F}^W$ and $\mathbb{R}$ by $\mathbb{R}^n$, $\mathbb{R}^{n\times m}$ or $\mathbb{S}^n$, where $\mathbb{S}^n$ is the set of symmetric $n\times n$ real matrices. If $M\in\mathbb{S}^n$ is positive definite (positive semidefinite) , we write $M>$ ($\geq$) $0.$
In our argument, $t$, $\omega$, ``almost surely'' and ``almost everywhere'', will be suppressed for simplicity in many circumstances, when no confusion occurs.

We now introduce the following scalar-valued linear stochastic differential equation (SDE):
\begin{align}
\label{state}
\begin{cases}
dX(t)=\left[A(t, \alpha_t)X(t)+B(t, \alpha_t)'u(t)\right]dt\\
\qquad\qquad+\left[C(t, \alpha_t)'X(t)+u(t)'D(t, \alpha_t)'\right]dW(t), \ t\in[0, T], \\
X(0)=x, \ \alpha_0=i_0,
\end{cases}
\end{align}
where $A(t, \omega, i), \ B(t, \omega, i), \ C(t, \omega, i), \ D(t, \omega, i)$ are all $\{\mathcal{F}^W_t\}_{t\geq 0}$-adapted processes of suitable sizes for $i\in\cM$, and $x\in\mathbb{R}$ is a given number.
Let $\Gamma\subset\mathbb{R}^m$ be a given closed cone, i.e., $\Gamma$ is closed, and if $u\in\Gamma$, then $\lambda u\in\Gamma$, for all $\lambda\geq 0$. It is the constraint set for controls.
The class of admissible controls is defined as the set
\begin{align*}
\mathcal{U}:=\Big\{u(\cdot)\in L^2_\mathcal{F}(0, T;\mathbb{R}^m)\;\Big|\; u(\cdot) \in\Gamma, \mbox{ a.e. a.s., and \eqref{state} has a unique strong solution}\Big\}.
\end{align*}
If $u(\cdot)\in\mathcal{U}$ and $X(\cdot)$ is the associated solution of \eqref{state}, then we refer to $(X(\cdot), u(\cdot))$ as an admissible pair.

Let us now state our stochastic linear quadratic optimal control problem (stochastic LQ problem, for short) as follows:
\begin{align}
\begin{cases}
\mathrm{Minimize} &\ J(x, i_0, u(\cdot))\\
\mbox{subject to} &\ (X(\cdot), u(\cdot)) \mbox{ admissible for} \ \eqref{state},
\end{cases}
\label{LQ}%
\end{align}
where the cost functional is given as the following quadratic form
\begin{align}\label{costfunctional}
J(x, i_0, u(\cdot)):=\mathbb{E}\left\{\int_0^T\Big(Q(t, \alpha_t)X(t)^2+u(t)'R(t, \alpha_t)u(t)\Big)dt
+G(\alpha_T)X(T)^2\right\}.
\end{align}

For $(x, i_0)\in\R\times\cM$, Problem \eqref{LQ} is said to be finite, if there exists $c\in\mathbb{R}$ such that
\begin{align*}
J(x, i_0, u(\cdot))\geq c, \quad \forall u(\cdot)\in\mathcal{U};
\end{align*}
and to be solvable, if there exists a control $u^*(\cdot)\in\mathcal{U}$ such that
\begin{align*}
-\infty<J(x, i_0, u^*(\cdot))\leq J(x, i_0, u(\cdot)), \quad \forall u(\cdot)\in\mathcal{U},
\end{align*}
in which case, $u^*(\cdot)$ is called an optimal control for Problem \eqref{LQ}.

Throughout this paper, we put the following assumptions on the coefficients.
\begin{assumption} \label{assump1}
For all $i\in\cM$,
\begin{align*}
\begin{cases}
A(t, \omega, i)\in L_{\mathcal{F}^W}^\infty(0, T;\mathbb{R}), \\
B(t, \omega, i)\in L_{\mathcal{F}^W}^\infty(0, T;\mathbb{R}^m), \\
C(t, \omega, i)\in L_{\mathcal{F}^W}^\infty(0, T;\mathbb{R}^n), \\
D(t, \omega, i)\in L_{\mathcal{F}^W}^\infty(0, T;\mathbb{R}^{n\times m}), \\
Q(t, \omega, i)\in L_{\mathcal{F}^W}^\infty(0, T;\mathbb{R}), \\
R(t, \omega, i)\in L_{\mathcal{F}^W}^\infty(0, T;\mathbb{S}^m), \\
G(\omega, i)\in L_{\mathcal{F}^W}^\infty(\Omega;\mathbb{R}).
\end{cases}
\end{align*}
\end{assumption}
By standard SDE theory, $\eqref{state}$ admits a unique solution $X(\cdot)\in L^2_{\mathcal{F}}(\Omega;C(0, T;\mathbb{R}))$ for any $u(\cdot)\in L^2_\mathcal{F}(0, T;\mathbb{R}^m)$ under Assumption \ref{assump1}.

The rest of the paper is devoted to the study of Problem \eqref{LQ} and its application in two portfolio selection problems.
\section{The extended stochastic Riccati equations}
To tackle Problem \eqref{LQ}, we need first to study two related multidimensional BSDEs.

For $\Lambda\in\mathbb{R}^n$ and $P\geq0$ with $PD(t, i)'D(t, i)+R(t, i)>0$, set
\begin{align*}
H_1(t, \omega, P, \Lambda, i)&=\inf_{v\in\Gamma}\big[v'(PD(t, i)'D(t, i)+R(t, i))v\\
&\qquad\qquad+2v'(PB(t, i)+PD(t, i)'C(t, i)+D(t, i)'\Lambda)\big], \\
H_2(t, \omega, P, \Lambda, i)&=\inf_{v\in\Gamma}\big[v'(PD(t, i)'D(t, i)+R(t, i))v\\
&\qquad\qquad-2v'(PB(t, i)+PD(t, i)'C(t, i)+D(t, i)'\Lambda)\big].
\end{align*}
Because $PD(t, i)'D(t, i)+R(t, i)$ is positive definite, $H_{1}$ and $H_{2}$ are well-defined, that is, $\R$-valued. Clearly, they are non-positive as $0\in\Gamma$.

\begin{remark}
For $P\geq0$ with $PD'D+R>0$, and $\Lambda\in\mathbb{R}^n$, we have the following estimates in the standard case ($R\geq\delta I_m$) (we drop the argument $(t,\omega,i)$ in this remark):
\begin{align*}
H_1(P, \Lambda)&=\inf_{v\in\Gamma}\big[v'(PD'D+R)v+2v'(PB+PD'C+D'\Lambda)\big] \\
&\geq \inf_{v\in\Gamma}\big[\delta |v|^2-2c(P+|\Lambda|)|v|\big]\\
&\geq -\frac{c^2(P+|\Lambda|)^2}{\delta},
\end{align*}
and in the singular case ($D'D\geq\delta I_m$):
\begin{align*}
H_1(P, \Lambda)&=\inf_{v\in\Gamma}\big[v'(PD'D+R)v+2v'(PB+PD'C+D'\Lambda)\big] \\
&\geq \inf_{v\in\Gamma}\big[\delta P|v|^2-2c(P+|\Lambda|)|v|\big]\\
&\geq -\frac{c^2(P+|\Lambda|)^2}{\delta P},
\end{align*}
where $c>0$, $\delta>0$ are two constants.
As $0\in\Gamma$,   we  have the trivial estimate $H_1( P, \Lambda)\leq 0$. Therefore
$|H_1(P, \Lambda)|\leq \frac{c^2(P+|\Lambda|)^2}{\delta}$ in the standard case, and
$H_1( P, \Lambda)\leq \frac{c^2(P+|\Lambda|)^2}{\delta P}$ in the singular case. This  reveals the quadratic nature of $H_1$.
Similar properties hold for  $H_2$.

If $\Gamma=\{\lambda \mathbf{e}_1|\lambda>0\}$ where $\mathbf{e}_1=(1,0,...,0)'\in\mathbb{R}^m$ is a unit vector, then $\Gamma$ is a ray (the interior is empty). In this case, we have an explicit expression for $H_1$:
\begin{align*}
H_1(P, \Lambda)&=\inf_{v\in\Gamma}\big[v'(PD'D+R)v+2v'(PB+PD'C+D'\Lambda)\big] \\
&=-\frac{((PB+PD'C+D'\Lambda)_1^-)^2}{(PD'D+R)_{11}}.
\end{align*}
If  $\Gamma=\{\lambda \mathbf{e}_1|\lambda>0\}\cup\{\lambda \mathbf{e}_2|\lambda>0\}$ where $\mathbf{e}_2=(0,1,0,...,0)'\in\mathbb{R}^m$ is also a unit vector, then $\Gamma$ is a set of two rays. And
\begin{align*}
H_1(P, \Lambda)&=\inf_{v\in\Gamma}\big[v'(PD'D+R)v+2v'(PB+PD'C+D'\Lambda)\big] \\
&=-\Big[\frac{((PB+PD'C+D'\Lambda)_1^-)^2}{(PD'D+R)_{11}}\vee\frac{((PB+PD'C+D'\Lambda)_2^-)^2}{(PD'D+R)_{22}}\Big],
\end{align*}
where $(PB+PD'C+D'\Lambda)_i$ is the $i$th element of the vector $(PB+PD'C+D'\Lambda)$ and $(PD'D+R)_{ij}$ is the $ij$th element of the $m\times m$ matrix $(PD'D+R)$, $i,j=1,...,m$.
If $\Gamma$ is a set of countable rays, we also have explicit expression for $H_1$ similarly.
\end{remark}

We introduce the following two multidimensional BSDEs (remind that the arguments $t$ and $\omega$ are suppressed):
\begin{align}
\label{P1}
\begin{cases}
dP_1(i)=-\Big[(2A(i)+C(i)'C(i))P_1(i)+2C(i)'\Lambda_1(i)+Q(i)\\
\qquad\qquad\qquad+H_1(P_1(i), \Lambda_1(i), i)+\sum\limits_{j=1}^{\ell}q_{ij}P_1(j)\Big]dt+\Lambda_1(i)'dW, \\
P_1(T, i)=G(i), \\
R(i)+P_1(i)D(i)'D(i)>0, \ \mbox{ for all $i\in\cM$};
\end{cases}
\end{align}
and
\begin{align}
\label{P2}
\begin{cases}
dP_2(i)=-\Big[(2A(i)+C(i)'C(i))P_2(i)+2C(i)'\Lambda_2(i)+Q(i)\\
\qquad\qquad\qquad+\ H_2(P_2(i), \Lambda_2(i), i)+\sum\limits_{j=1}^{\ell}q_{ij}P_2(j)\Big]dt+\Lambda_2(i)'dW, \\
P_2(T, i)=G(i), \\
R(i)+P_2(i)D(i)'D(i)>0, \ \mbox{ for all $i\in\cM$.}
\end{cases}
\end{align}

\begin{remark}\label{remark:sym1}
If $\Gamma$ is symmetric, namely, $-v\in\Gamma$ whenever $v\in\Gamma$, then
\( H_1(P, \Lambda, i)=H_2(P, \Lambda, i).\) Therefore BSDEs \eqref{P1} and \eqref{P2} coincide, and if each BSDE admits a unique solution, then $P_1=P_2$. In particular, if there is no control constraint, i.e. $\Gamma=\mathbb{R}^m$, then both $H_1$ and $H_2$ equal to
\begin{align*}
-[PB(i)'+(PC(i)+\Lambda)'D(i)](R(i)+PD(i)'D(i))^{-1}[PB(i)+D(i)'(PC(i)+\Lambda)].
\end{align*}
\end{remark}

BSDEs \eqref{P1} and \eqref{P2} are referred to as the \emph{extended stochastic Riccati equations} (ESREs). When
$\Gamma=\mathbb{R}^m$ and $\ell=1$ (namely, there is no control constraint or regime switching), then they degenerate to the stochastic Riccati equation studied in \cite{KT}.

\begin{definition}
A vector process $(P(i), \ \Lambda(i))_{i=1}^{\ell}$ is called a solution of the multidimensional BSDE \eqref{P1}, if it satisfies \eqref{P1}, and $(P(i), \ \Lambda(i))\in L^\infty_{\mathcal{F}^W}(0, T; \mathbb {R})\times L^{2}_{\mathcal F^{W}}(0, T;\mathbb{R}^{n})$ for all $i\in\cM$. The solution of BSDE system \eqref{P2} is defined similarly.
\end{definition}

Usually one would seek the solutions of \eqref{P1} and \eqref{P2} in the space $L^\infty_{\mathcal{F}^W}(0, T; \mathbb {R})\times L^{2}_{\mathcal F^{W}}(0, T;\mathbb{R}^{n})$ for all $i\in\cM$. This space, however, is not precise  enough in the proof of uniqueness.

In fact, the second part solutions $\Lambda$ of \eqref{P1} and \eqref{P2} turn out to be in the class of martingales of bounded mean oscillation, briefly called \emph{BMO martingales}. To give proper definitions of their solutions, here we recall some facts about BMO martingales; see Kazamaki \cite{Ka}. The process $\int_0^\cdot \Lambda(s)'dW(s)$ is a BMO martingale if and only if there exists a constant $c>0$ such that
\[\mathbb{E}\Bigg[\int_\tau^T|\Lambda(s)|^2ds\Big|\mathcal F_\tau^W\Bigg]\leq c\]
for all $\{\mathcal{F}_t^W\}_{t\geq 0}$-stopping times $\tau\leq T$.
The Dol$\acute{\mathrm{e}}$ans-Dade stochastic exponential
$$\mathcal E(\int_0^\cdot \Lambda(s)'dW(s))$$
of a BMO martingale $\int_0^\cdot \Lambda(s)'dW(s)$ is a uniformly integrable martingale. Moreover, if $\int_0^\cdot \Lambda(s)'dW(s)$ and $\int_0^\cdot Z(s)'dW(s)$ are both BMO martingales, then under the probability measure $\widetilde{\mathbb{P}}$ defined by $\frac{d\widetilde{\mathbb{P}}}{d\mathbb{P}}\big|_{\mathcal{F}_T}=\mathcal E \big(\int_0^T Z(s)'dW(s)\big)$, $\widetilde W(\cdot):=W(\cdot)-\int_0^\cdot Z(s)ds$ is a standard Brownian motion, and $\int_0^\cdot \Lambda(s)'d\widetilde W(s)$ is a BMO martingale.

The following space plays an important role in our argument
\begin{align*}
\BMO&=\bigg\{\Lambda \in L^{2}_{\mathcal F^{W}}(0, T;\mathbb{R}^{n}) \;\bigg|\; \int_0^\cdot\Lambda(s)'dW(s) \mbox{ is a BMO martingale on $[0, T]$}\bigg\}.
\end{align*}
\subsection{Solutions to ESREs: Existence and uniqueness}
In this section, we address ourselves to the solvability of \eqref{P1} and \eqref{P2}.

Both \eqref{P1} and \eqref{P2} are highly nonlinear multidimensional BSDEs. There are several results on the solvability of stochastic Riccati equations or quadratic BSDE systems (see, e.g., Hu and Zhou \cite{HZ}, Kohlmann and Tang \cite{KT}, Tang \cite{Ta}, Hu and Tang \cite{HT}). But up to our knowledge, no existing results could be directly applied to \eqref{P1} or \eqref{P2}, because they violate \emph{both} the standard Lipschitz condition and the quadratic growth condition.

The following comparison theorem for multidimensional BSDEs can be found in \cite{HP} (one can find a concise version in \cite{HLT}). We shall use it frequently in the study of BSDEs \eqref{P1} and \eqref{P2}.
We provide the sketch of its proof in \hyperref[appn]{Appendix}  for the reader's convenience.

\begin{lemma}
\label{comparison}
Suppose $(Y(i), Z(i))_{i}^{\ell}$, $(\overline Y(i), \overline Z(i))_{i}^{\ell}$ satisfy the following two $\ell$-dimensional BSDEs, respectively:
\begin{align*}
Y(t, i)=\xi(i)+\int_t^T f(s, Y(s, i), Y(s, -i), Z(s, i), i)ds-\int_t^T Z(s, i)'dW(s), \ \mbox{ for all $i\in\cM$;}
\end{align*}
and
\begin{align*}
\overline Y(t, i)=\overline\xi(i)+\int_t^T \overline f(s, \overline Y(s, i), \overline Y(s, -i), \overline Z(s, i), i)ds-\int_t^T \overline Z(s, i)'dW(s), \ \mbox{ for all $i\in\cM$, }
\end{align*}
where $Y(s, -i)=(Y(s, 1), \ldots, Y(s, i-1), Y(s, i+1), \ldots, Y(s, \ell))$.
Also suppose that, for all $i\in\cM$,
\begin{enumerate}
\item $\xi(i), \ \overline\xi(i)\in L^2_{\mathcal{F}^W}(\Omega;\mathbb{R})$, and $\xi(i)\leq\overline\xi(i)$;
\item there exists a constant $c>0$ such that
\[|f(s, y, z, i)-f(s, \overline y, \overline z, i)|\leq c(|y-\overline y|+|z-\overline z|), \]
for any $z, \overline z\in\mathbb{R}^n$, $y=(y(i), y(-i))$, $\overline y=(\overline y(i), \overline y(-i))\in\mathbb{R}^{\ell}$;

\item $f(s, y, z, i)$ is nondecreasing in $y(j)$, for every $i\neq j \in\cM$; and
\item $ f(s, \overline Y(s, i), \overline Y(s, -i), \overline Z(s, i), i)\leq \overline f(s, \overline Y(s, i), \overline Y(s, -i), \overline Z(s, i), i)$.

\end{enumerate}
Then $Y(t, i)\leq \overline Y(t, i)$ for a.e. $t\in[0, T]$ and all $i\in\cM$.
\end{lemma}
We emphasis that the above lemma requires the global Lipschitz condition, which is not satisfied in some cases in our below discussion.
\par
We now prove the existence and uniqueness for the solution of BSDE \eqref{P1}. That for \eqref{P2} are similar, so we omit the details. We will treat two cases separately: (1) standard case, in which $R(i)$ is (uniformly) positive definite; (2) singular case, in which $R(i)$ is positive semidefinite but $G(i)$ and $D(i)'D(i)$ are (uniformly) positive definite. Here ``singular" means that the control weight matrix $R(i)$ in the cost functional \eqref{costfunctional} could be probably a  singular matrix.

\begin{theorem} [standard case]
\label{Riccatistandard}
Assume that $G(i)\geq0, \ Q(i)\geq0$, and $R(i)\geq\delta I_m$ with some deterministic constant $\delta>0$, for a.e. $ t\in[0, T]$ and all $i\in\cM$.
Then BSDE \eqref{P1} admits a unique solution $(P(i), \ \Lambda(i))_{i=1}^{\ell}$ such that $P(i)\geq 0$, for  all $i\in\cM$.
\end{theorem}
\begin{proof}
Existence. For $i\in\cM$, $P\in\mathbb{R}^{\ell}$, and $\Lambda\in\mathbb{R}^{n\times \ell}$, set
\begin{align*}
\overline f(t, P, \Lambda, i)=(2A(i)+C(i)'C(i)+q_{ii})P(i)+2C(i)'\Lambda(i)+Q(i)+\sum_{j\neq i}q_{ij}P(j).
\end{align*}
As $\overline f$ is linear in $P$ and $\Lambda$, there exists a unique solution $(\overline P(i), \ \overline\Lambda(i))_{i=1}^{\ell}$ to the corresponding BSDE with the generator $\overline f$ and terminal value $G$. By Assumption \ref{assump1}, there exists a constant $c>0$, such that
\begin{align*}
2A(i)+C(i)'C(i)+\max_{k, j\in \cM}|q_{kj}|\leq c, \ Q(i)\leq c, \ G(i)\leq c, \ \mbox{ for a.e. $t\in[0, T]$ and all } i\in\cM.
\end{align*}
Hereafter, we shall use $c$ to represent a generic positive constant independent of $i$, $n$ and $t$, which can be different from line to line.

The following $\ell$-dimensional BSDE
\begin{align*}
\begin{cases}
dP(i)=-\Big[c\sum_{j=1}^{\ell}P(j)+2C(i)\Lambda(i)+c\Big]dt+\Lambda(i)'dW, \\
P(i, T)=c, \ \mbox{ for all $i\in\cM$, }
\end{cases}
\end{align*}
admits a unique solution $\left(\frac{(c\ell+1)e^{c\ell(T-t)}-1}{\ell}, 0\right)_{i=1}^{\ell}$. By Lemma \ref{comparison}, we have
\[
\overline P(t, i)\leq \frac{(c\ell+1)e^{c\ell(T-t)}-1}{\ell}\leq M, \ \mbox{ for a.e. $t\in[0, T]$ and all } i\in\cM.
\]
where $M=\frac{(c\ell+1)e^{c\ell T}-1}{l}$.

For $k\geq1$, $(t, P, \Lambda)\in[0, T]\times\mathbb R\times\mathbb{R}^n$, $ i\in\cM$, define
\[
H^k(t, P, \Lambda, i)=\sup_{\tilde P\in\mathbb R, \tilde\Lambda\in\mathbb R^n}\Big\{H_1 (t, \tilde P, \tilde\Lambda, i)-k|P-\tilde P|-k|\Lambda-\tilde\Lambda|\Big\}.
\]
Then it is non-positive and uniformly Lipschitz in $(P, \Lambda)$, and decreasingly approaches to $H_1(t, P, \Lambda, i)$ as $k$ goes to infinite.

The following BSDE
\begin{align*}
\begin{cases}
dP^k(i)=-\Big[
\overline f(P^k, \Lambda^k, i)+H^k(P^k(i), \Lambda^k(i), i)\Big]dt+\Lambda^k(i)'dW, \\
P^k(i, T)=G(i), \ \mbox{ for all $i\in\cM$, }
\end{cases}
\end{align*}
is an $\ell$-dimensional BSDE with a Lipschitz generator, so it admits a unique solution, denoted by $\big(P^k(i), \Lambda^k(i)\big)_{i=1}^{\ell}$.
Notice that $H^k(t, 0, 0, i)=0, \ Q\geq0, \ G\geq 0$, and
\[\overline f(t, P, \Lambda, i)+H^k(t,P(i), \Lambda(i), i) \leq \overline f(t, P, \Lambda, i), \]
then by Lemma \ref{comparison}, we have
\begin{align*}
\label{Pbound}
0\leq P^k(t, i)\leq\overline P(t, i)\leq M,
\end{align*}
and $P^k(t, i)$ is decreasing in $k$, for each $i\in\cM$.

Let $P(t, i)=\lim\limits_{k\rightarrow\infty}P^k(t, i)$, $i\in\cM$.
It is important to note that we can regard $\big(P^k(i), \Lambda^k(i)\big)$ as the solution of a scalar-valued quadratic BSDE for each $i\in\cM$. Thus by Lemma 9.6.6 in \cite{CZ}, there exists a process $\Lambda\in L^{2}_{\mathcal F^{W}}(0, T;\mathbb{R}^{n\times \ell})$ such that $(P, \Lambda)$ is a solution to BSDE \eqref{P1}. We have now established the existence of the solution.\\

Next, let us prove the uniqueness.

Step 1: For any solution $(P(i), \ \Lambda(i))\in L^\infty_{\mathcal{F}^W}(0, T; \mathbb {R}^+)\times L^{2}_{\mathcal F^{W}}(0, T;\mathbb{R}^{n})$ of \eqref{P1}, we have a more precise estimate $\Lambda(i)\in\BMO$, for all $i\in\cM$.

Actually applying It\^{o}'s formula to $P(i)^2$, we get, for any $\{\mathcal F^W_t\}_{t\geq 0}$ stopping time $\tau\leq T$,
\begin{align*}
&\E\Big[\int_\tau^T|\Lambda(i)|^2ds\Big|\mathcal{F}^W_\tau\Big]=\E[G(i)^2|\mathcal{F}^W_\tau]-P(t,i)^2
+\E\Big[\int_\tau^T2P(i)\Big[(2A(i)+|C(i)|^2)P(i)\\
&\qquad\qquad\qquad\qquad\qquad+2 C(i)'\Lambda(i)+Q(i)+H_1(P(i),\Lambda(i),i)+\sum_{i=1}^\ell q_{ij}P(j)\Big]ds\Big|\mathcal{F}^W_\tau\Big].
\end{align*}
Note that $H_1\leq0$, Assumption \ref{assump1} and $P$ is uniformly bounded, then
\begin{align*}
\E\Big[\int_\tau^T|\Lambda(s,i)|^2ds\Big|\mathcal{F}^W_\tau\Big]
&\leq c+\E\Big[\int_\tau^T\Big[c+2 c|\Lambda(s,i)|\Big]ds\Big|\mathcal{F}^W_\tau\Big]\\
&\leq c+\frac{1}{2}\E\Big[\int_\tau^T|\Lambda(s,i)|^2ds\Big|\mathcal{F}^W_\tau\Big].
\end{align*}
Thus $\Lambda(i)\in\BMO$, for all $i\in\cM$.\\

Step 2: Log transformation of the BSDE \eqref{P1}.

Suppose $(P(i), \ \Lambda(i))_{i=1}^{\ell}$, $(\tilde P(i), \ \tilde\Lambda(i))_{i=1}^{\ell}$ are two solutions of $\eqref{P1}$. Then there exists a constant $M>0$ such that $0\leq P(i)$, $\tilde P(i)\leq M$, and $\int_0^\cdot\Lambda(s, i)'dW(s)$, $\int_0^\cdot\tilde\Lambda(s, i)'dW(s)$ are BMO-martingales, for all $i\in\cM$.

For every $i\in\cM$, define processes
\begin{align*}
&(U(t, i), V(t, i))=\left(\ln (P(t, i)+a), \frac{\Lambda(t, i)}{P(t, i)+a}\right), \\
&(\tilde U(t, i), \tilde V(t, i))=\left(\ln (\tilde P(t, i)+a), \frac{\tilde \Lambda(t, i)}{\tilde P(t, i)+a}\right), \ \mbox{ for } t\in[0, T],
\end{align*}
where $a>0$ is a constant to be determined later.
Then $(U(i), V(i))$, $(\tilde U(i), \tilde V(i))\in L^\infty_{\mathcal{F}^W}(0, T; \mathbb {R})\times \BMO$, for all $ i\in\cM$. Furthermore, by It\^{o}'s formula, $(U(i), V(i))_{i=1}^{\ell}$ satisfy the following multidimensional BSDE:
\begin{align*}
\begin{cases}
dU(i)=-\Big[(2A(i)+C(i)'C(i))(1-ae^{-U(i)})+2C(i)'V(i)+Q(i)e^{-U(i)}\\
\qquad\qquad\qquad+\tilde H(U(i), V(i), i)+\frac{1}{2}V(i)'V(i)+\sum\limits_{j=1}^{\ell}q_{ij}e^{U(j)-U(i)}\Big]dt+V(i)'dW, \\
U(T, i)=\ln (G(i)+a), \ \mbox{ for all $i\in\cM$, }
\end{cases}
\end{align*}
where
\begin{align*}
\tilde H(U, V, i)&=\inf_{v\in\Gamma}\Big[v'((1-ae^{-U})D(i)'D(i)+R(i)e^{-U})v\\
&\qquad\qquad+2v'((1-ae^{-U})(B(i)+D(i)'C(i))+D(i)'V)\Big].
\end{align*}
Similar for $(\tilde U(i), \tilde V(i))_{i=1}^{\ell}$.

As $0\leq P(i)\leq M$, thus $e^{-U(i)}=\frac{1}{P(i)+a}\geq\frac{1}{M+a}$ and $1-ae^{-U(i)}=\frac{P(i)}{P(i)+a}\in[0, 1)$. Similar inequalities hold when $U(i)$ is replaced by $\tilde U(i)$. By Assumption \ref{assump1} and $R\geq\delta I_m$, there exist constant $c>0$ such that
\begin{align*}
&\ \ \ \ v'((1-ae^{-U(i)})D(i)'D(i)+R(i)e^{-U(i)})v\\
&\qquad\qquad+2v'((1-ae^{-U(i)})(B(i)+D(i)'C(i))+D(i)'V(i))\\
&\geq v'R(i)e^{-U(i)}v+2v'((1-ae^{-U(i)})(B(i)+D(i)'C(i))+D(i)'V(i))\\
&\geq \frac{\delta}{M+a} |v|^2-c(1+|V(i)|)|v|.
\end{align*}
Hence if $|v|>\frac{c(M+a)}{\delta}(1+|V|):=c(1+|V|)$, then
\[\frac{\delta}{M+a} |v|^2-c(1+|V|)|v|>0\geq \tilde H(U, V, i), \]
for $(U, V)=(U(t, i), V(t, i))$ and $(\tilde U(t, i), \tilde V(t, i))$. Thus,
\begin{align*}
\tilde H(U, V, i)&=\inf_{\substack{v\in\Gamma\\ |v|\leq c(1+|V|)}}\Big[v'((1-ae^{-U})D(i)'D(i)+R(i)e^{-U})v\nonumber\\
&\qquad\qquad\qquad\quad+2v'((1-ae^{-U})(B(i)+D(i)'C(i))+D(i)'V)\Big].
\end{align*}

Step 3: Estimate the difference between $U(i)$ and $\tilde U(i)$.

Set $\bar U(i)=U(i)-\tilde U(i), \ \bar V(i)=V(i)-\tilde V(i)$, for $i\in\cM$. Then $(\bar U(i), \ \bar V(i))_{i}^{\ell}$ satisfy the following BSDE:
\begin{align*}
\begin{cases}
d\bar U(i)=-\Big[(Q(i)-2aA(i)-aC(i)'C(i))(e^{-U(i)}-e^{-\tilde U(i)})+2C(i)'\bar V(i)\\
\qquad\qquad\qquad+\tilde H(U(i), V(i), i)-\tilde H(\tilde U(i), \tilde V(i), i)+\frac{1}{2}(V(i)+\tilde V(i))'\bar V(i) \\
\qquad\qquad\qquad+\sum\limits_{j=1}^{\ell}q_{ij}(e^{U(j)-U(i)}-e^{\tilde U(j)-\tilde U(i)})\Big]dt+\bar V(i)'dW, \\
\bar U(T,i)=0, \ \mbox{ for all $i\in\cM$.}
\end{cases}
\end{align*}
Applying It\^{o}'s formula to $\bar U(i)^2$, we deduce that
\begin{align*}
\bar U(t, i)^2&=\int_t^T\Big\{2\bar U(i)\Big[(Q(i)-2aA(i)-aC(i)'C(i))(e^{-U(i)}-e^{-\tilde U(i)})+2C(i)'\bar V(i)\\
&\qquad\qquad+\frac{1}{2}(V(i)+\tilde V(i))'\bar V(i)+\sum\limits_{j=1}^{\ell}q_{ij}(e^{U(j)-U(i)}-e^{\tilde U(j)-\tilde U(i)})\Big]-\bar V(i)'\bar V(i)\\
&\qquad\qquad+2\bar U(i)(\tilde H(U(i), V(i), i)-\tilde H(\tilde U(i), \tilde V(i), i))\Big\}ds-\int_t^T2\bar U(i)\bar V(i)'dW\\
&:=\int_t^T\Big[L(i)+2\bar U(i)\sum\limits_{j=1}^{\ell}q_{ij}(e^{U(j)-U(i)}-e^{\tilde U(j)-\tilde U(i)})\Big]ds-\int_t^T2\bar U(i)\bar V(i)'dW.
\end{align*}
Let us now estimate $\bar U(i)\left(\tilde H(U(i), V(i), i)-\tilde H(\tilde U(i), \tilde V(i), i)\right)$.
Here we encounter the major technique issue of the paper.
If $\tilde H(U, V, i)$ was decreasing in $U$, then
\begin{align*}
&\quad\bar U(i)\left(\tilde H(U(i), V(i), i)-\tilde H(\tilde U(i), \tilde V(i), i)\right) \\
&=(U(i)-\tilde U(i))\left(\tilde H(U(i), V(i), i)-\tilde H(\tilde U(i), V(i), i)+\tilde H(\tilde U(i), V(i), i)-\tilde H(\tilde U(i), \tilde V(i), i)\right) \\
&\leq (U(i)-\tilde U(i))\left(\tilde H(\tilde U(i), V(i), i)-\tilde H(\tilde U(i), \tilde V(i), i)\right),
\end{align*}
which would be  growth in $\bar V(i)$ quadratically at most. Unfortunately we do not have the monotonicity of $\tilde H(U, V, i)$ in $U$.
To overcome this difficultly, we treat the quadratic term and linear term separately. Let
\begin{align*}
\cH(v, \tilde U, U, V, i)&=v'((1-ae^{-\tilde U})D(i)'D(i)+R(i)e^{-\tilde U})v\\
&\qquad\qquad+2v'((1-ae^{-U})(B(i)+D(i)'C(i))+D(i)'V).
\end{align*}
Then rewrite $\tilde H$ in terms of $\cH$,
\begin{align*}
&\qquad\tilde H(U(i), V(i), i)-\tilde H(\tilde U(i), \tilde V(i), i)\\
&=\inf_{\substack{v\in\Gamma\\ |v|\leq c(1+|V(i)|)}} \cH(v, U, U, V, i)-\inf_{\substack{v\in\Gamma\\ |v|\leq c(1+|\tilde V(i)|)}} \cH(v, \tilde U, \tilde U, \tilde V, i).
\end{align*}
The above optimal values will not change when we enlarge the optimization region, so,
after inserting two zero-sum terms, we get
\begin{align} \label{ineqH1}
&\qquad\;\tilde H(U(i), V(i), i)-\tilde H(\tilde U(i), \tilde V(i), i) \nonumber\\
&=\inf_{\substack{v\in\Gamma\\ |v|\leq c(1+|V(i)|+|\tilde V(i)|)}} \cH(v, U, U, V, i)-\inf_{\substack{v\in\Gamma\\ |v|\leq c(1+|V(i)|+|\tilde V(i)|)}}\cH(v, \tilde U, U, V, i)\nonumber\\
&\ \ \ \ \ \ \ \ \ +\inf_{\substack{v\in\Gamma\\ |v|\leq c(1+|V(i)|+|\tilde V(i)|)}}\cH(v, \tilde U, U, V, i)-\inf_{\substack{v\in\Gamma\\ |v|\leq c(1+|V(i)|+|\tilde V(i)|)}} \cH(v, \tilde U, \tilde U, \tilde V, i).
\end{align}

Let $a>0$ be a sufficiently small constant such that $R(i)-aD(i)'D(i)>0$ for all $i\in\cM$. Then the map
\[x\mapsto v'((1-ae^{-x})D(i)'D(i)+R(i)e^{-x})v, \quad x\in\R, \]
is decreasing for every $i\in\cM$. Therefore,
\begin{align} \label{ineqH2}
(U(i)-\tilde U(i))\Big(\inf_{\substack{v\in\Gamma\\ |v|\leq c(1+|V(i)|+|\tilde V(i)|)}} \cH(v, U, U, V, i)-\inf_{\substack{v\in\Gamma\\ |v|\leq c(1+|V(i)|+|\tilde V(i)|)}} \cH(v, \tilde U, U, V, i)\Big)\leq 0.
\end{align}
On the other hand, by the boundedness of $U$ and $\tilde U$,
\begin{align} \label{ineqH3}
&\qquad\; \bigg|\inf_{\substack{v\in\Gamma\\ |v|\leq c(1+|V(i)|+|\tilde V(i)|)}}\cH(v, \tilde U, U, V, i)-\inf_{\substack{v\in\Gamma\\ |v|\leq c(1+|V(i)|+|\tilde V(i)|)}}\cH(v, \tilde U, \tilde U, \tilde V, i)\bigg| \nonumber\\
&\leq \sup_{\substack{v\in\Gamma\\ |v|\leq c(1+|V(i)|+|\tilde V(i)|)}} \big|\cH(v, \tilde U, U, V, i)-\cH(v, \tilde U, \tilde U, \tilde V, i)\big|\nonumber\\
&=\sup_{\substack{v\in\Gamma\\ |v|\leq c(1+|V(i)|+|\tilde V(i)|)}}\Big|2v'\Big(a(e^{-U}-e^{-\tilde U})(B(i)+D(i)'C(i))+D(i)'\bar V\Big)\Big|\nonumber\\
&\leq c \sup_{\substack{ |v|\leq c(1+|V(i)|+|\tilde V(i)|)}}|v|(|\bar U(i)|+|\bar V(i)|)\nonumber\\
&\leq c (1+|V(i)|+|\tilde V(i)|)(|\bar U(i)|+|\bar V(i)|).
\end{align}
Using \eqref{ineqH1}, \eqref{ineqH2} and \eqref{ineqH3}, we get
\begin{align*}
&\quad\bar U(i)\left(\tilde H(U(i), V(i), i)-\tilde H(\tilde U(i), \tilde V(i), i)\right)\\
&\leq c|\bar U(i)|(1+|V(i)|+|\tilde V(i)|)(|\bar U(i)|+|\bar V(i)|)\\
&=:c(1+|V(i)|+|\tilde V(i)|)\bar U(i)^2+c\beta(i)'\bar U(i)\bar V(i),
\end{align*}
where $\beta(i)$ is an $\mathcal F^W_t$-adapted process such that $|\beta(i)|\leq 1+|V(i)|+|\tilde V(i)|$. Then
\begin{align*}
L(i)&\leq 2\bar U(i)\Big[(Q(i)-2aA(i)-aC(i)'C(i))(e^{-U(i)}-e^{-\tilde U(i)})+2C(i)'\bar V(i)\nonumber\\
&\qquad+\frac{1}{2}(V(i)+\tilde V(i))'\bar V(i)\Big]-\bar V(i)'\bar V(i)\\
&\qquad+c(1+|V(i)|+|\tilde V(i)|)\bar U(i)^2+c\beta(i)' \bar U(i)\bar V(i)\\
&\leq c(1+|V(i)|+|\tilde V(i)|)\bar U(i)^2+2\bar U(i)\bar V(i)'\Big(2C(i)+\frac{1}{2}(V(i)+\tilde V(i))+c\beta(i)\Big),
\end{align*}
for all $i\in\cM$.

For each fixed $i\in\cM$, let us introduce the processes
\begin{align*}
J(t, i)=\exp\left(\int_0^tc(1+|V(i)|+|\tilde V(i)|)ds\right),
\end{align*}
and
\begin{align*}
N(t, i)=\mathcal{E}\left(\int_0^t\Big(2C(i)+\frac{1}{2}(V(i)+\tilde V(i))+c\beta(i)\Big)'dW(s)\right).
\end{align*}
Note that $N(t, i)$ is a uniformly integrable martingale. Thus
\[\widetilde W^i(t):=W(t)-\int_0^t\Big(2C(i)+\frac{1}{2}(V(i)+\tilde V(i))+c\beta(i)\Big)ds, \]
is a Brownian motion under the probability $\widetilde{\mathbb P}^{i}$ defined by
\begin{align*}
\frac{d\widetilde{\mathbb P}^{i}}{d\mathbb P}\Bigg|_{\mathcal{F}^W_T}=N(T, i).
\end{align*}
It\^{o}'s formula gives us, for any $\{\mathcal F^W_t\}_{t\geq 0}$-stopping time $\tau$ such that $0\leq t\leq\tau\leq T$,
\begin{align*}
&\quad\; J(t, i)N(t, i)\bar U(t, i)^2\\
&\leq J(\tau, i) N(\tau, i)\bar U(\tau, i)^2+2\int_t^\tau J(i)N(i)\bar U(i)\sum\limits_{j=1}^{\ell}q_{ij}(e^{U(j)-U(i)}-e^{\tilde U(j)-\tilde U(i)})ds\\
&\quad\;\;-\int_t^\tau\left(J(i)N(i)\bar U(i)^2\Big(2C(i)+\frac{1}{2}(V(i)+\tilde V(i))+c\beta(i)\Big)+2J(i)N(i)\bar U(i)\bar V(i)\right)'dW.
\end{align*}
Let us consider, for $n\geq 1$, the stopping time
\begin{align*}
\tau_n&=\inf\Bigg\{u\geq t:\int_t^u\bigg|J(i)N(i)\bar U(i)^2\Big(2C(i)+\frac{1}{2}(V(i)+\tilde V(i))+c\beta(i)\Big)\\
&\qquad\qquad\qquad\quad\qquad+2J(i)N(i)\bar U(i)\bar V(i)\bigg|^2ds\geq n\Bigg\}\wedge T.
\end{align*}
We get from the previous equation and the arithmetic-mean and geometric-mean inequality (AM-GM inequality),
\begin{align}
\label{JNinequality}
\bar U(t, i)^2&\leq \E\Bigg[\frac{J(\tau_n, i) N(\tau_n, i)\bar U(\tau_n, i)^2}{J(t, i)N(t, i)}\nonumber\\
&\qquad\qquad+2\int_t^{\tau_{n}} \frac{J(s, i)N(s, i)}{J(t, i)N(t, i)}\bar U(i)\sum\limits_{j=1}^{\ell}q_{ij}(e^{U(j)-U(i)}-e^{\tilde U(j)-\tilde U(i)})ds\;\bigg|\;\mathcal{F}^W_t\Bigg]\nonumber\\
&\leq \E\Bigg[\frac{J(\tau_n, i) N(\tau_n, i)\bar U(\tau_n, i)^2}{J(t, i)N(t, i)}+c\int_t^{\tau_{n}} \frac{J(s, i)N(s, i)}{J(t, i)N(t, i)}\sum\limits_{j=1}^{\ell}\bar U(s, j)^2ds\;\bigg|\;\mathcal{F}^W_t\Bigg]\nonumber\\
&=\widetilde\E^i\Bigg[\frac{J(\tau_n, i) \bar U(\tau_n, i)^2}{J(t, i)}+c\int_t^{\tau_{n}} \frac{J(s, i)}{J(t, i)}\sum\limits_{j=1}^{\ell}\bar U(s, j)^2ds\;\bigg|\;\mathcal{F}^W_t\Bigg]\nonumber\\
&=\widetilde\E^i\Bigg[\bar U(\tau_n, i)^2\exp\left(\int_t^{\tau_n} c(1+|V(i)|+|\tilde V(i)|)ds\right)\nonumber\\
&\qquad\qquad+c\int_t^{\tau_n} \sum\limits_{j=1}^{\ell}\bar U(s, j)^2\exp\left(\int_t^{s} c(1+|V(i)|+|\tilde V(i)|)du\right)ds\;\bigg|\;\mathcal{F}^W_t\Bigg]\nonumber\\
&\leq c \widetilde\E^i\Bigg[\bigg(\bar U(\tau_n, i)^2+\int_t^{T} \sum\limits_{j=1}^{\ell}\bar U(s, j)^2ds\bigg)\exp\left(\int_t^{T} c(|V(i)|+|\tilde V(i)|)du\right)\;\bigg|\;\mathcal{F}^W_t\Bigg]
\end{align}
where $\widetilde\E^i$ is the expectation w.r.t. the probability measure $\widetilde{\mathbb P}^i$.
By the AM-GM inequality, we see
\begin{align*}
&\quad\; \widetilde\E^i\left[\exp\left(\int_t^{T} c(|V(i)|+|\tilde V(i)|)ds\right)\Bigg|\mathcal{F}^W_t\right] \\
&\leq \widetilde\E^i\left[\exp\left(\int_t^{T} \Big(\varepsilon_i (|V(i)|+|\tilde V(i)|)^2+\frac{c^2}{4\varepsilon_i}\Big)ds\right)\Bigg|\mathcal{F}^W_t\right],
\end{align*}
which is finite by the John-Nirenberg inequality when $\varepsilon_i>0$ is sufficient small.
Using boundedness of $\bar U$ and the dominated convergent theorem, sending $n$ to infinity in \eqref{JNinequality} gives
\begin{align*}
\bar U(t, i)^2
&\leq c \widetilde\E^i\Bigg[\bigg(\int_t^{T} \sum\limits_{j=1}^{\ell}\bar U(s, j)^2ds\bigg)\exp\left(\int_t^{T} c(|V(i)|+|\tilde V(i)|)du\right)\;\bigg|\;\mathcal{F}^W_t\Bigg]\\
&\leq c\widetilde\E^i\Bigg[\exp\left(\int_t^{T} c(|V(i)|+|\tilde V(i)|)ds\right)\Bigg|\mathcal{F}^W_t\Bigg]\int_t^{T} \sum\limits_{j=1}^{\ell} E(s, j)ds\\
&\leq c\int_t^{T} \sum\limits_{j=1}^{\ell} E(s, j)ds,
\end{align*}
where
\[E(t, i)=\underset{\omega\in\Omega}{\esssup} \ {\bar U(t, i)^2}.\]
Taking essential supreme on both sides, we deduce
\[E(t, i)\leq c \int_t^T \sum_{j=1}^{\ell}E(s, j)ds.\]
Thus
\[0\leq \sum_{j=1}^{\ell}E(t, j)\leq cl\int_t^T \sum_{j=1}^{\ell}E(s, j)ds.\]
We infer from Gronwall's inequality that $\sum_{j=1}^{\ell}E(t, j)=0$, so
$\bar U(t, i)=0$ for a.e. $t\in[0, T]$ and all $i\in\cM$. This completes the proof of the uniqueness.
\end{proof}

\begin{theorem}[singular case]
\label{Riccatisingular}
Assume that $Q(i)\geq0$, $R(i)\geq0, \ G(i)\geq\delta$, and $D(i)'D(i)\geq\delta I_m$ with some deterministic constant $\delta>0$, for a.e. $t\in[0, T]$ and all $i\in\cM$. Then BSDE \eqref{P1} admits a unique solution such that $P(i)\geq c$ for some constant $c>0$, for all $i\in\cM$.
\end{theorem}
\begin{proof}
This case is relatively easy to deal with. We will present the main idea only. Details are left for the interested readers.
\par
Set, for $P\in\mathbb{R}^{\ell}_+$ and $\Lambda\in\mathbb{R}^{n\times \ell}$.
\begin{align*}
&\underline f(t, P, \Lambda, i)=(2A(i)+C(i)'C(i)+q_{ii})P(i)+2C(i)'\Lambda(i)+Q(i)+H_{1}(t, P(i), \Lambda(i), i).
\end{align*}
The corresponding $\ell$-dimensional BSDE with the generator $\underline f$ and terminal value $G$ is decoupled,
then by Theorem 4.2 of \cite{HZ}, there exists a solution $(\underline P(i), \ \underline\Lambda(i))\in L^\infty_{\mathcal{F}^W}(0, T; \mathbb {R})\times \BMO$, such that $\underline P(i)\geq c$ with some constant $c>0$, for all $i\in\cM$.

Let $g:\mathbb{R}^+\rightarrow [0, 1]$ be a smooth truncation function satisfying $g(x)=0$ for $x\in[0, \frac{1}{2}c]$, and $g(x)=1$ for $x\in[c, +\infty)$.
Repeat the argument of the proof of Theorem \ref{Riccatistandard}, the following BSDE:
\begin{align*}
\begin{cases}
dP^k(i)=-\Big[
\overline f(P^k, \Lambda^k, i)+H_{c}^k(P^k(i), \Lambda^k(i), i)\Big]dt+\Lambda^k(i)'dW, \\
P^k(i, T)=G(i), \ \mbox{ for all $i\in\cM$, }
\end{cases}
\end{align*}
has a solution, $(P^k(i), \ \Lambda^k(i))_{i=1}^{\ell}$, where $\overline f(t, P, \Lambda, i)$ is the same as in the proof of Theorem \ref{Riccatistandard} and
\[ H^k_c(t, P, \Lambda, i)=\sup_{\tilde P\in\mathbb R, \tilde\Lambda\in\mathbb R^n}\Big\{H_{1} (t, \tilde P, \tilde\Lambda, i)g(P)-k|P-\tilde P|-k|\Lambda-\tilde\Lambda|\Big\}. \]
Notice that
\[
\underline f(t, P, \Lambda, i)\leq \overline f(t,P, \Lambda, i)+H_{c}^k(t,P(i), \Lambda(i), i) \leq \overline f(t, P, \Lambda, i),
\]
then
\begin{align*}
c\leq\underline P(i)\leq P^k(i)\leq\overline P(i)\leq M.
\end{align*}
After taking limit, it gives $P(i)=\lim_{k}P^k(i)\geq c$ so that $g(P(i))=1$ and
\[H^k_c(t, P(i), \Lambda(i), i)=H^k(t, P(i), \Lambda(i), i).\]
By this, we proved the existence.
\par
To prove the uniqueness, using $P(i)\geq c>0$, we can repeat the proof of Theorem \ref{Riccatistandard} with $a=0$. In this case \eqref{ineqH3} is replaced by
\begin{align*}
&\qquad\;\bigg|\inf_{\substack{v\in\Gamma\\ |v|\leq c(1+|V(i)|+|\tilde V(i)|)}}\cH(v, \tilde U, U, V, i)-\inf_{\substack{v\in\Gamma\\ |v|\leq c(1+|V(i)|+|\tilde V(i)|)}} \cH(v, \tilde U, \tilde U, \tilde V, i)\bigg| \nonumber\\
&\leq \sup_{\substack{v\in\Gamma\\ |v|\leq c(1+|V(i)|+|\tilde V(i)|)}}\Big|2v' D(i)'\bar V(i)\Big| \nonumber\\
&\leq c (1+|V(i)|+|\tilde V(i)|) |\bar V(i)|,
\end{align*}
so that we do not need the John-Nirenberg inequality in the proof.
\end{proof}
\section{Solution to the LQ problem \eqref{LQ}}
In this section, we solve the LQ problem \eqref{LQ} explicitly in terms of the solutions of BSDEs \eqref{P1} and \eqref{P2}.

For $P\geq0$ with $PD(t, i)'D(t, i)+R(t, i)>0$, and $\Lambda\in\mathbb{R}^n$, define
\begin{align*}
\hat v_1(t, \omega, P, \Lambda, i)&=\argmin_{v\in\Gamma}\big[v'(PD(t, i)'D(t, i)+R(t, i))v\\
&\qquad\qquad\quad+2v'(PB(t, i)+PD(t, i)'C(t, i)+D(t, i)'\Lambda)\big], \\
\hat v_2(t, \omega, P, \Lambda, i)&=\argmin_{v\in\Gamma}\big[v'(PD(t, i)'D(t, i)+R(t, i))v\\
&\qquad\qquad\quad-2v'(PB(t, i)+PD(t, i)'C(t, i)+D(t, i)'\Lambda)\big].
\end{align*}
Similar to Remark \ref{remark:sym1}, we have
\begin{remark}\label{remark:sym2}
If $\Gamma$ is symmetric, then
\( \hat v_1(t,P, \Lambda, i)=-\hat v_2(t,P, \Lambda, i).\) In particular, if $\Gamma=\mathbb{R}^m$, then
\begin{align*}
\hat v_2(t, P, \Lambda, i)=(PD(t, i)'D(t, i)+R(t, i))^{-1}(PB(t, i)+PD(t, i)'C(t, i)+D(t, i)'\Lambda).
\end{align*}
\end{remark}

\begin{theorem}\label{LQcontrol}
Under conditions assumed  in Theorem \ref{Riccatistandard} or Theorem \ref{Riccatisingular},
the LQ problem \eqref{LQ} admits an optimal control, as a feedback function of the time $t$, the state $X$, and the market regime $i$,
\begin{align}
\label{opticon}
u^*(t, X, i)=\hat v_1(t, P_1(t, i), \Lambda_1(t, i), i)X^++\hat v_2(t, P_2(t, i), \Lambda_2(t, i), i)X^-.
\end{align}
Moreover, the corresponding optimal value is
\begin{align*}
\min_{u\in\mathcal{U}}J(x, i_0, u(\cdot))=P_1(0, i_0)(x^+)^2+P_2(0, i_0)(x^-)^2,
\end{align*}
where $(P_1(i), \ \Lambda_1(i))_{i=1}^{\ell} \ ((P_2(i), \ \Lambda_2(i))_{i=1}^{\ell})$ are the unique solutions of \eqref{P1} (\eqref{P2}).
\end{theorem}

\begin{lemma}
Under the conditions of Theorem \ref{LQcontrol}, the feedback control $u^*$ defined by \eqref{opticon} is an admissible control for Problem \eqref{LQ}. \end{lemma}
\begin{proof}
By definition, we can see that $\hat v_1(t, P_1, \Lambda_1, i) $, $\hat v_2(t, P_2, \Lambda_2, i)\in\Gamma$, so is $u^*(t, X, i)$.
It is only left to show that $u^*(t, X(t), \alpha_t)\in L^2_\mathcal{F}(0, T;\mathbb{R}^m)$.

Substituting \eqref{opticon} into the state process \eqref{state}, we have
\begin{align}
\label{SDEstate}
\begin{cases}
dX(t)=\big[A(t, \alpha_t)X(t)+B(t, \alpha_t)'(\hat v_1(t, P_1, \Lambda_1, \alpha_t)X^+(t)+\hat v_2(t, P_2, \Lambda_2, \alpha_t)X^-(t))\big]dt\\
\qquad \qquad \;+\big[C(t, \alpha_t)'X(t)+(\hat v_1(t, P_1, \Lambda_1, \alpha_t)X^+(t)\\
\qquad\qquad\qquad \;+\hat v_2(t, P_2, \Lambda_2, \alpha_t)X^-(t))'D(t, \alpha_t)'\big]dW(t), \\
X(0)=x, \ \alpha_0=i_0,
\end{cases}
\end{align}
Similarly to the proofs of Theorems \ref{Riccatistandard} and \ref{Riccatisingular}, we know there are constants $c_1>0, \ c_2>0$, for a.e. $t\in[0, T]$ and all $i\in\cM$, such that
\begin{align*}
&\ \ \ \ v'(P_1(i)D(i)'D(i)+R(i))v
+2v'(P_1(i)B(i)+P_1(i)D(i)'C(i)+D(i)'\Lambda_1(i))\\
&\geq c_1|v|^2-c_2(|P_1(i)|+|\Lambda_1(i)|)|v|.
\end{align*}
Notice that $c_1|v|^2-c_2(|P_1(i)|+|\Lambda_1(i)|)|v|>0 \geq H_1(t, P_1, \Lambda_1) $, if $|v|>\frac{c_2}{c_1}(|P_1(i)|+|\Lambda_1(i)|):=c(|P_1(i)|+|\Lambda_1(i)|)$, thus
\[
|\hat v_1(t, P_1(i), \Lambda_1(i), i)|\leq c(|P_1(i)|+|\Lambda_1(i)|).
\]
Similarly, we have
\[
|\hat v_2(t, P_2, \Lambda_2, i)|\leq c(|P_2(i)|+|\Lambda_2(i)|).
\]
From Theorems \ref{Riccatistandard} and \ref{Riccatisingular}, we know that
$(P_1(i), \ \Lambda_1(i))$, $(P_2(i), \ \Lambda_2(i))\in L^2_{\mathcal{F}^W}(0, T; \mathbb {R})\times \BMO$ for all $i\in\cM$.
By the basic theorem on pp. 756-757 of Gal'chuk \cite{Ga}, the SDE \eqref{SDEstate} has a unique strong solution.
Furthermore,
\begin{align*}
|u^*(t, X(t), \alpha_t)| \leq c (|P_1(t, \alpha_t)|+|\Lambda_1(t, \alpha_t)|+|P_2(t, \alpha_t)|+|\Lambda_2(t, \alpha_t)|)|X(t)|.
\end{align*}
As $X(t)$ is continuous, it is bounded on $[0, T]$. Hence we guarantee that
\begin{align*}
\int_0^T |u^*(t, X(t), \alpha_t)|^2dt < \infty.
\end{align*}

Applying It\^{o}'s lemma to $P_1(t, \alpha_t)X^+(t)^2$ and $P_2(t, \alpha_t)X^-(t)^2$, (Here we use It\^{o}'s lemma for the Markovian chain and we refer to \cite{HLT}), we have
\begin{align*}
& \quad\; P_1(t, \alpha_t)X^+(t)^2+P_2(t, \alpha_t)X^-(t)^2\\
&\qquad+\int_0^t\Big[u^*(s, X(s), \alpha_s)'R(s, \alpha_s)u^*(s, X(s), \alpha_s)+Q(s, \alpha_s)X(s)^2\Big]ds\\
&=P_1(t, \alpha_t)X^+(t)^2+P_2(t, \alpha_t)X^-(t)^2\\
&\qquad
+\int_0^t\Big[X^+(s)^2\hat v_1(s, P_1, \Lambda_1, \alpha_s)'R(s, \alpha_s)\hat v_1(s, P_1, \Lambda_1, \alpha_s)\\
&\qquad\qquad\qquad+X^-(s)^2\hat v_2(s, P_2, \Lambda_2, \alpha_s)'R(s, \alpha_s)\hat v_2(s, P_2, \Lambda_2, \alpha_s)+Q(s, \alpha_s)X(s)^2\Big]ds\\
&=P_1(0, i_0)(x^+)^2+P_2(0, i_0)(x^-)^2\\
&\qquad+\int_0^t\Big\{X^+(s)^2( 2P_1(s, \alpha_s)(C(s, \alpha_s)+D(s, \alpha_s)\hat v_1(s, P_1, \Lambda_1, \alpha_s))'
+\Lambda_1(s, \alpha_s)')\\
&\qquad\qquad\qquad+X^-(s)^2 (2P_2(s, \alpha_s)(C(s, \alpha_s)-D(s, \alpha_s)\hat v_2(s, P_2, \Lambda_2, \alpha_s))'\\
&\qquad\qquad\qquad\qquad+\Lambda_2(s, \alpha_s)')\Big\}dW(s)\\
&\qquad+\int_0^t\Big\{X^+(s)^2\sum_{j, j'\in\mathcal{M}}(P_1(s, j)-P_1(s, j'))I_{\{\alpha_{s-}=j'\}}\\
&\qquad\qquad\qquad+X^-(s)^2\sum_{j, j'\in\mathcal{M}}(P_2(s, j)-P_2(s, j'))I_{\{\alpha_{s-}=j'\}}\Big\}d\tilde N_s^{j'j},
\end{align*}
where $(N^{j'j})_{j'j\in\mathcal{M}}$ are independent Poisson processes each with intensity $q_{j'j}$, and $\tilde N_t^{j'j}=N_t^{j'j}-q_{j'j}, \ t\geq0$ are the corresponding compensated Poisson martingales under the filtration $\mathcal{F}$.

Noting that $X(t)$ is bounded on $[0, T]$, the stochastic integrals in the last equation are local martingales. Thus there exists an increasing sequence of stopping times $\tau_n$ such that $\tau_n\uparrow+\infty$ as $n\rightarrow+\infty$ such that
\begin{align}
\label{stop}
& \quad\;\mathbb{E}\Bigg[P_1(\iota\wedge\tau_n, \alpha_{\iota\wedge\tau_n})X^+(\iota\wedge\tau_n)^2
+P_2(\iota\wedge\tau_n, \alpha_{\iota\wedge\tau_n})X^-(\iota\wedge\tau_n)^2 \nonumber\\
&\qquad\quad+\int_0^{\iota\wedge\tau_n}\Big[u^*(s, X(s), \alpha_s)'R(s, \alpha_s)u^*(s, X(s), \alpha_s)+Q(s, \alpha_s)X(s)^2\Big]ds\Bigg] \nonumber\\
&=P_1(0, i_0)(x^+)^2+P_2(0, i_0)(x^-)^2,
\end{align}
for any stopping time $\iota\leq T$.

For the standard case, we have
\begin{align*}
\delta\mathbb{E}\int_0^{T\wedge\tau_n}|u^*(s, X(s), \alpha_s)|^2ds\leq P_1(0, i_0)(x^+)^2+P_2(0, i_0)(x^-)^2,
\end{align*}
where $\delta>0$ is given in Theorem \ref{Riccatistandard}. Letting $n\to \infty$, it follows from the monotone convergent theorem that
$u^*(t, X(t), \alpha_t)\in L^2_\mathcal{F}(0, T;\mathbb{R}^m)$.

For the singular case, there exists a constant $c>0$, such that $P_1(i)$, $P_2(i)\geq c$ for all $i\in\cM$ by Theorem \ref{Riccatisingular}.
Then from \eqref{stop}, we have
\begin{align*}
c\mathbb{E}\big[X(\iota\wedge\tau_n)^2\big]&\leq\mathbb{E} \Big[P_1(\iota\wedge\tau_n, \alpha_{\iota\wedge\tau_n})X^+(\iota\wedge\tau_n)^2
+P_2(\iota\wedge\tau_n, \alpha_{\iota\wedge\tau_n})X^-(\iota\wedge\tau_n)^2\Big]\\
&\leq P_1(0, i_0)(x^+)^2+P_2(0, i_0)(x^-)^2.
\end{align*}
Letting $n\to \infty$, it follows from Fatou's lemma that
\begin{align}
\label{bound}
\E \Big[X(\iota\wedge T)^2\Big]\leq c,
\end{align}
for any stopping time $\iota\leq T$.
This further implies
\begin{align}
\label{bound2}
\E \int_0^{\iota\wedge T}X(s)^2ds\leq \int_0^{T}\E\left[X(s)^2\right]ds\leq cT.
\end{align}
By It\^{o}'s Lemma, we have
\begin{align*}
X(t)^2&=x^2+\int_0^t \Big[(2A(s, \alpha_s)+C(s, \alpha_s)'C(s, \alpha_s))X(s)^2 \\
&\qquad\qquad\qquad+2X(s)\big((B(s, \alpha_s)+D(s, \alpha_s)'C(s, \alpha_s))'u^*(s, X(s), \alpha_s)\big)\\
&\qquad\qquad\qquad+u^*(s, X(s), \alpha_s)'D(s, \alpha_s)'D(s, \alpha_s)u^*(s, X(s), \alpha_s)\Big]ds \\
&\qquad+\int_0^t2X(s)(C(s, \alpha_s)'X(s)+u^*(s, X(s), \alpha_s)'D(s, \alpha_s)')dW(s).
\end{align*}
Because $X(t)$ is continuous, it follows that
\begin{align*}
2X(s)(C(s, \alpha_s)'X(s)+u^*(s, X(s), \alpha_s)'D(s, \alpha_s)')\in L^{2, \;\mathrm{loc}}_{\mathcal{F}}(0, T;\mathbb R^n).
\end{align*}
Therefore, there exists an increasing localizing sequence $\tau_n\uparrow\infty$ as $n\rightarrow\infty$, such that
\begin{align*}
&\ \ \ \ x^2+\E \int_0^{T\wedge\tau_n}u^*(s, X(s), \alpha_s)'D(s, \alpha_s)'D(s, \alpha_s)u^*(s, X(s), \alpha_s)ds\\
&=\E\Big[ X(T\wedge\tau_n)^2\Big]-\E \int_0^{T\wedge\tau_n} \Big[(2A(s, \alpha_s)+C(s, \alpha_s)'C(s, \alpha_s))X(s)^2\\
&\qquad\qquad\qquad\qquad\qquad+2X(s)\big((B(s, \alpha_s)+D(s, \alpha_s)'C(s, \alpha_s))'u^*(s, X(s), \alpha_s)\big)\Big]ds.
\end{align*}
Let $\delta>0$ be given in Theorem \ref{Riccatisingular}. By \eqref{bound} and \eqref{bound2}, the above by the AM-GM inequality leads to
\begin{align*}
&\quad\;\delta \E \int_0^{T\wedge\tau_n}|u^*(s, X(s), \alpha_s)|^2ds \leq c+c\E \int_0^{T\wedge\tau_n}\Big[X(s)^2+2|X(s)| |u^*(s, X(s), \alpha_s)|\Big]ds\\
&\leq c+c(1+\frac{2c}{\delta})\E \int_0^{T\wedge\tau_n}X(s)^2ds+\frac{\delta}{2}\E \int_0^{T\wedge\tau_n} |u^*(s, X(s), \alpha_s)|^{2}ds\\
&\leq c+\frac{\delta}{2}\E \int_0^{T\wedge\tau_n} |u^*(s, X(s), \alpha_s)|^{2}ds.
\end{align*} After rearrangement, it follows from the monotone convergent theorem that
\begin{align*}
\E \int_0^{T}|u^*(s, X(s), \alpha_s)|^2ds \leq c.
\end{align*}
This completes the proof.
\end{proof}

We are now ready to prove Theorem \ref{LQcontrol}.

\begin{proof}
For any $u(\cdot)\in\mathcal{U}$,
applying It\^{o}'s formula to $P_1(t, \alpha_t)X^+(t)^2$ and $P_2(t, \alpha_t)X^-(t)^2$, we have
\begin{align*}
& \ \ \ \ P_1(t, \alpha_t)X^+(t)^2+P_2(t, \alpha_t)X^-(t)^2\\
&=P_1(0, i_0)(x^+)^2+P_2(0, i_0)(x^-)^2\\
&\quad+\int_0^t \bigg\{I_{\{X(s)\geq0\}}P_1(s, \alpha_s)u(s)'D(s, \alpha_s)'D(s, \alpha_s)u(s)\\
&\qquad\qquad\;+2X^+(s)P_1(s, \alpha_s)C(s, \alpha_s)'D(s, \alpha_s)u(s)\\
&\qquad\qquad\;+2X^+(s)P_1(s, \alpha_s)B(s, \alpha_s)'u(s)+2X(s)^+\Lambda_1(s, \alpha_s)'D(s, \alpha_s)u(s)\\
&\qquad\qquad\;-X^+(s)^2[Q(s, \alpha_s)+H_1(s, P_1(s, \alpha_s), \Lambda_1(s, \alpha_s), \alpha_s)]\\
&\qquad\qquad\;+I_{\{X(s)<0\}}P_2(s, \alpha_s)u(s)'D(s, \alpha_s)'D(s, \alpha_s)u(s)\\
&\qquad\qquad\;-2X^-(s)P_2(s, \alpha_s)C(s, \alpha_s)'D(s, \alpha_s)u(s)\\
&\qquad\qquad\;-2X^-(s)P_2(s, \alpha_s)B(s, \alpha_s)'u(s)-2X^-(s)\Lambda_2(s, \alpha_s)'D(s, \alpha_s)u(s)\\
&\qquad\qquad\;-X^-(s)^2[Q(s, \alpha_s)+H_2(s, P_2(s, \alpha_s), \Lambda_2(s, \alpha_s), \alpha_s)]\bigg\}ds\\
&\quad+\int_0^t\bigg\{ 2P_1(s, \alpha_s)(C(s, \alpha_s)'X^+(s)^2+u_s'D(s, \alpha_s)'X^+(s))
+X^+(s)^2\Lambda_1(s, \alpha_s)'\\
&\qquad\qquad\;+2P_2(s, \alpha_s)(C(s, \alpha_s)'X^-(s)^2-u_s'D(s, \alpha_s)'X^-(s))
+X^-(s)^2\Lambda_2(s, \alpha_s)'\bigg\}dW(s)\\
&\quad+\int_0^t\bigg\{X^+(s)^2\sum_{j, j'\in\mathcal{M}}(P_1(s, j)-P_1(s, j'))I_{\{\alpha_{s-}=j'\}}\\
&\qquad\qquad\;+X^-(s)^2\sum_{j, j'\in\mathcal{M}}(P_2(s, j)-P_2(s, j'))I_{\{\alpha_{s-}=j'\}}\bigg\}d\tilde N_s^{j'j},
\end{align*}
where $(N^{j'j})_{j'j\in\mathcal{M}}$ are independent Poisson processes each with intensity $q_{j'j}$, and $\tilde N_t^{j'j}=N_t^{j'j}-q_{j'j}, \ t\geq0$ are the corresponding compensated Poisson martingales under the filtration $\mathcal{F}$.

Note that $X(t)$ is continuous, the last two terms in the above equation are local martingales. Therefore, there exists an increasing localizing sequence of stopping times $\tau_n\uparrow+\infty$ as $n\rightarrow+\infty$ such that
\begin{align*}
& \ \ \ \ \mathbb{E}\left[P_1(T\wedge\tau_n, \alpha_{T\wedge\tau_n})X^+(T\wedge\tau_n)^2
+P_2(T\wedge\tau_n, \alpha_{T\wedge\tau_n})X^-(T\wedge\tau_n)^2\right]\\
&=P_1(0, i_0)(x^+)^2+P_2(0, i_0)(x^-)^2\\
&\quad+\mathbb{E}\int_0^{T\wedge\tau_n} \bigg\{I_{\{X(s)\geq0\}}P_1(s, \alpha_s)u(s)'D(s, \alpha_s)'D(s, \alpha_s)u(s)\\
&\qquad\qquad\;+2X^+(s)P_1(s, \alpha_s)C(s, \alpha_s)'D(s, \alpha_s)u(s)\\
&\qquad\qquad\;+2X^+(s)P_1(s, \alpha_s)B(s, \alpha_s)'u_s+2X^+(s)\Lambda_1(s, \alpha_s)'D(s, \alpha_s)u(s)\\
&\qquad\qquad\;+I_{\{X(s)<0\}}P_2(s, \alpha_s)u(s)'D(s, \alpha_s)'D(s, \alpha_s)u(s)\\
&\qquad\qquad\;-2X^-(s)P_2(s, \alpha_s)C(s, \alpha_s)'D(s, \alpha_s)u(s)\\
&\qquad\qquad\;-2X^-(s)P_2(s, \alpha_s)B(s, \alpha_s)'u(s)-2X^-(s)\Lambda_2(s, \alpha_s)'D(s, \alpha_s)u(s)\\
&\qquad\qquad\;-X^+(s)^2H_1(s, P_1(s, \alpha_s), \Lambda_1(s, \alpha_s), \alpha_s)-X^-(s)^2H_2(s, P_2(s, \alpha_s), \Lambda_2(s, \alpha_s), \alpha_s)\\
&\qquad\qquad\;-X(s)^2Q(s, \alpha_s)\bigg\}ds.
\end{align*}
where we combined the two terms involving $Q$.
After rearrangement and combining similar terms,
\begin{align}
\label{square}
& \ \ \ \ \mathbb{E}\bigg[P_1(T\wedge\tau_n, \alpha_{T\wedge\tau_n})X^+(T\wedge\tau_n)^2
+P_2(T\wedge\tau_n, \alpha_{T\wedge\tau_n})X^-(T\wedge\tau_n)^2\nonumber\\
&\qquad\qquad\qquad\qquad
+\int_0^{T\wedge\tau_n} \Big(Q(s, \alpha_s)X(s)^2+u_s'R(s, \alpha_s)u(s)\Big)ds\bigg]\nonumber\\
&=P_1(0, i_0)(x^+)^2+P_2(0, i_0)(x^-)^2+\mathbb{E}\int_0^{T\wedge\tau_n} \phi(s, X(s), u(s), \alpha_s) ds,
\end{align}
where
\begin{align*}
&\phi(s, X(s), u(s), \alpha_s)\\
&=u(s)'\Big(R(s, \alpha_s)+I_{\{X(s)\geq0\}}P_1(s, \alpha_s)D(s, \alpha_s)'D(s, \alpha_s)\\
&\qquad\qquad\;+I_{\{X(s)<0\}}P_2(s, \alpha_s)D(s, \alpha_s)'D(s, \alpha_s)\Big)u(s)\\
&\quad\;+2X^+(s)(P_1(s, \alpha_s)C(s, \alpha_s)'D(s, \alpha_s)+P_1(s, \alpha_s)B(s, \alpha_s)'
+\Lambda_1(s, \alpha_s)'D(s, \alpha_s))u(s)\\
&\quad\;-2X^-(s)(P_2(s, \alpha_s)C(s, \alpha_s)'D(s, \alpha_s)+P_2(s, \alpha_s)B(s, \alpha_s)'+\Lambda_2(s, \alpha_s)'D(s, \alpha_s))u(s)\\
&\quad\;-X^+(s)^2H_1(s, P_1(s, \alpha_s), \Lambda_1(s, \alpha_s), \alpha_s)-X^-(s)^2H_2(s, P_2(s, \alpha_s), \Lambda_2(s, \alpha_s), \alpha_s).
\end{align*}
Define an $\mathcal{F}_t$-adapted process
\begin{align*}
v(t)=\begin{cases}
\frac{u(t)}{|X(t)|}, \ &\mbox{ if } \ |X(t)|>0;\\
\ \ 0, \ &\mbox{ if } \ X(t)=0.
\end{cases}
\end{align*}
Notice that $\Gamma$ is a cone, so the process $v$ is valued in $\Gamma$. If $X(s)\geq0$, then
\begin{align*}
\phi(s, X(s), u(s), \alpha_s)&=X(s)^2\Big\{v(s)'\Big(R(s, \alpha_s)+P_1(s, \alpha_s)D(s, \alpha_s)'D(s, \alpha_s)\Big)v(s)\\
&\qquad\qquad\;+2(P_1(s, \alpha_s)C(s, \alpha_s)'D(s, \alpha_s)+P_1(s, \alpha_s)B(s, \alpha_s)'\\
&\qquad\qquad\;+\Lambda_1(s, \alpha_s)'D(s, \alpha_s))v(s)-H_1(s, P_1(s, \alpha_s), \Lambda_1(s, \alpha_s), \alpha_s)\Big\}.
\end{align*}
By the definition of $H_1(t, P, \Lambda, i)$, this is non-negative. Similarly, we have $\phi(s, X(s), u(s), \alpha_s)\geq0$ when $X(s)<0$. Hence, it follows from \eqref{square} that
\begin{align*}
& \quad\; \mathbb{E}\bigg[P_1(T\wedge\tau_n, \alpha_{T\wedge\tau_n})X^+(T\wedge\tau_n)^2
+P_2(T\wedge\tau_n, \alpha_{T\wedge\tau_n})X^-(T\wedge\tau_n)^2\nonumber\\
&\qquad\quad
+\int_0^{T\wedge\tau_n} \Big(Q(s, \alpha_s)X(s)^2+u_s'R(s, \alpha_s)u(s)\Big)ds\bigg]\nonumber\\
&\geq P_1(0, i_0)(x^+)^2+P_2(0, i_0)(x^-)^2.
\end{align*}
It is not hard to verify $\mathbb E\bigg[\sup\limits_{t\in[0, T]}X(t)^2\bigg]<\infty$ by standard theory of SDE. Let $n\rightarrow\infty$, by the dominated convergence and monotone convergence theorems, we have
\begin{align*}
&\quad\mathbb{E}\left\{\int_0^T\Big(Q(t, \alpha_t)X(t)^2+u(t)'R(t, \alpha_t)u(t)\Big)dt
+G(\alpha_T)X(T)^2\right\}\\
& \geq P_1(0, i_0)(x^+)^2+P_2(0, i_0)(x^-)^2,
\end{align*}
where the equality holds at \eqref{opticon}.
\end{proof}

\section{Application to mean-variance portfolio selection problems}

Consider a financial market consisting of a risk-free asset (the money market
instrument or bond) whose price is $S_{0}$ and $m$ risky securities (the
stocks) whose prices are $S_{1}, \ldots, S_{m}$. And assume $m\leq n$, i.e. the number of risky securities is no more than the dimension of the Brownian motion.
The asset prices $S_k$, $k=0, 1, \ldots, m, $ are driven by SDEs:
\begin{align*}
\begin{cases}
dS_0(t)=r(t, \alpha_t)S_{0}(t)dt, \\
S_0(0)=s_0,
\end{cases}
\end{align*}
and
\begin{align*}
\begin{cases}
dS_k(t)=S_k(t)\Big(\mu_k(t, \alpha_t)dt+\sum\limits_{j=1}^n\sigma_{kj}(t, \alpha_t)dW_j(t)\Big), \\
S_k(0)=s_k,
\end{cases}
\end{align*}
where, for every $k=1, \ldots, m$ and $i\in\cM$, $r(t, i)$ is the interest rate process, $\mu_k(t, i)$ and $\sigma_k(t, i):=(\sigma_{k1}(t, i), \ldots, \sigma_{kn}(t, i))$ are the appreciation rate process and volatility rate process of the $k$th risky security corresponding to a market regime $\alpha_t=i$.

Define the appreciate vector
\begin{align*}
\mu(t, i)=(\mu_1(t, i), \ldots, \mu_m(t, i))',
\end{align*}
and volatility matrix
\begin{align*}
\sigma(t, i)=
\left(
\begin{array}{c}
\sigma_1(t, i)\\
\vdots\\
\sigma_m(t, i)\\
\end{array}
\right)
\equiv (\sigma_{kj}(t, i))_{m\times n}, \ \text{for}\ \text{each} \ i\in\cM.
\end{align*}
In the rest part of this paper, we shall assume
$r(\cdot, \cdot, i)$, $\mu_k(\cdot, \cdot, i)$, $\sigma_{kj}(\cdot, \cdot, i)\in L^\infty_{\mathcal F^W}(0, T;\mathbb R)$, for all $k=1, \ldots, m$, $j=1, \ldots, n$, and $i\in\cM$. Also there exists a constant $\delta>0$ such that
$\sigma(t, i)\sigma(t, i)'\geq \delta I_m$ for a.e. $t\in[0, T]$ and all $i\in\cM$.

A small investor, whose actions cannot affect the asset prices, will decide at every time
$t\in[0, T]$ what amount $\pi_j(t)$ of his wealth to invest in the $j$th risky asset, $j=1, \ldots, m$. The vector process $\pi(\cdot):=(\pi_1(\cdot), \ldots, \pi_m(\cdot))'$ is called a portfolio of the investor. Then the investor's self-financing wealth process $X(\cdot)$ corresponding to a portfolio $\pi(\cdot)$ is the unique strong solution of the SDE:
\begin{align}
\label{wealth}
\begin{cases}
dX(t)=[r(t, \alpha_t)X(t)+\pi(t)'b(t, \alpha_t)]dt+\pi(t)'\sigma(t, \alpha_t)dW(t), \\
X(0)=x, \ \alpha_0=i_0,
\end{cases}
\end{align}
where $b(t, \alpha_t):=\mu(t, \alpha_t)-r(t, \alpha_t)\mathbf{1}_{m}$ and $\mathbf{1}_{m}$ is the $m$-dimensional vector with all entries being one.

\begin{remark}
\label{incompleteness}
This is an incomplete financial market model. The incompleteness comes from two rescources.
On one hand, the number of risky securities may be less than the dimension of the Brownian motion so that one can not perfectly hedge the risk; On the other hand, the Markov chain $\alpha_t$, which is independent of the Brownian motion, brings another market uncertainty.
\end{remark}

The admissible portfolio set is defined as
\begin{align*}
\mathcal U=\Big\{\pi\in L^2_{\mathcal F}(0, T;\mathbb R^m)\mid \pi(\cdot)\in\Gamma\mbox{ a.s. a.e.} \Big\}.
\end{align*}
For any $\pi\in \mathcal{U}$, the SDE \eqref{wealth} has a unique strong solution.
In the following two subsections, we will consider two different portfolio constraint sets: $\Gamma=\R^{m}$ and $\Gamma=\R_{+}^{m}$, respectively. Economically speaking, the former means there is no trading constraint; while the later means no-shorting is allowed in the market.

\begin{remark}
Our argument can be applied to consider general closed, not necessarily convex, cone portfolio constraint.
\end{remark}

\par
For a given expectation level $z\in\mathbb{R}$, the investor's problem is to
\begin{align}
\mathrm{Minimize}&\quad \mathrm{Var}(X(T))=\E\big[(X(T)-z)^{2}\big]%
, \nonumber\\
\mathrm{ s.t.} &\quad
\begin{cases}
\E (X(T))=z, \\
\pi\in \mathcal{U}.
\end{cases}
\label{optm}%
\end{align}

To deal with the constraint $\E(X(T))=z$, we introduce a Lagrange
multiplier $-2\lambda\in\mathbb{R}$ and obtain the following \emph{relaxed}
optimization problem:
\begin{align}\label{optmun}
\mathrm{Minimize} &\quad {\mathbb{E}}(X(T)-z)^{2}-2\lambda({\mathbb{E}}%
X(T)-z)={\mathbb{E}}(X(T)-(\lambda+z))^{2}-\lambda^{2}=:\hat{J}(\pi, \lambda), \\
\mathrm{s.t.} &\quad \pi\in \mathcal{U}.\nonumber%
\end{align}
Because Problem \eqref{optm} is a convex optimization problem, Problems \eqref{optm} and \eqref{optmun} are linked by the Lagrange duality theorem (see Luenberger \cite{Lu})
\begin{align}\label{duality}
\min_{\pi\in\mathcal{U}, \E(X(T))=z}\mathrm{Var}(X(T)%
)=\max_{\lambda\in\mathbb{R}}\min_{\pi\in\mathcal{U}}\hat{J}(\pi, \lambda).
\end{align}
This allows us to solve Problem \eqref{optm} by a two-step procedure: First solve the relaxed problem \eqref{optmun}, then find a $\lambda^{*}$ to maximize $\min_{\pi\in\mathcal{U}}\hat{J}(\pi, \lambda)$.
\par

\subsection{Feasibility of Problem \eqref{optm}}
We shall say that the mean-variance problem \eqref{optm} is feasible for a given $z$ if there is a portfolio $\pi\in \mathcal{U}$ which satisfies the target constraint $\E(X(T))=z$.

Economically speaking, it is irrational to consider those portfolios with expected returns $z<x\E e^{\int_0^Tr(t, \alpha_t)dt}$. Although general cases can be considered, for notation simplicity, we will focus on $z\geq x\E e^{\int_0^Tr(t, \alpha_t)dt}$.

Define
\[\dualgamma:=\{y\in\R^{m}\mid x'y\leq 0 \mbox{ for all $x\in\Gamma$}\}.\]
The following result gives an equivalent condition for the feasibility of \eqref{optm}.
\begin{theorem}[feasibility]\label{theoremfeasible}
Let  $(\psi(i), \xi(i))\in L^2_{\mathcal{F}^W}(0, T;\mathbb{R})\times L^2_{\mathcal{F}^W}(0, T;\mathbb{R}^n)$ for all $i\in\cM$ be the unique solution of the following linear multidimensional BSDE:
\begin{align*}
\begin{cases}
d\psi(t, i)=-\Big(r(t, i)\psi(t, i)+\sum\limits_{j=1}^{\ell} q_{ij}\psi(t, j)\Big)dt+\xi(t, i)'dW(t), \\
\psi(T, i)=1, \ \mbox{ for all $i\in\cM$.}
\end{cases}
\end{align*}
Then the mean-variance problem \eqref{optm} is feasible for any $z\geq x\E e^{\int_0^Tr(t, \alpha_t)dt}$ if and only if
\begin{equation}
\label{feasible}
\int_0^T \mathbb{P}\Big(\psi(t, \alpha_t)b(t, \alpha_t)+\sigma(t, \alpha_t)\xi(t, \alpha_t)\notin\dualgamma\Big) dt>0.
\end{equation}
\end{theorem}
\begin{proof}
For any $\pi\in\mathcal{U}$ and real number $\beta\geq 0$, we construct a portfolio $\pi^\beta(t):=\beta\pi(t)$.
Then $\pi^\beta\in\mathcal{U}$.
Let $X^\beta$ be the wealth process corresponding to $\pi^\beta$. Then $X^\beta(t)=X_0(t)+\beta X_1(t)$, where $X_{0}$ follows the SDE
\begin{align*}
\begin{cases}
dX_0(t)=r(t, \alpha_t)X_0(t)dt, \\
X_0(0)=x, \ \alpha_0=i_0,
\end{cases}
\end{align*}
and $X_{1}$ follows the SDE
\begin{align*}
\begin{cases}
dX_1(t)=[r(t, \alpha_t)X_1(t)+\pi(t)'b(t, \alpha_t)]dt+\pi(t)'\sigma(t, \alpha_t)dW(t), \\
X_1(0)=0, \ \alpha_0=i_0.
\end{cases}
\end{align*}
Applying It\^{o}'s lemma to $X_1(t)\psi(t, \alpha_t)$, we have
\begin{align}
\E \big(X^\beta(T)\big)&=\E \big(X_0(T)\big)+\beta \E \big(X_1(T)\big)\nonumber\\
&=x \E e^{\int_0^Tr(t, \alpha_t)dt}+\beta\E \int_0^T\pi(t)'(\psi(t, \alpha_t)b(t, \alpha_t)+\sigma(t, \alpha_t)\xi(t, \alpha_t))dt. \label{expectation}
\end{align}

We now prove the `` if '' part. Let $f$ be a measurable function such that $f(y)\in\Gamma$, $|f(y)|\leq1$ and
\[y'f(y)=\max_{x\in\Gamma, \; |x|\leq1} x'y, \]
for any $y\in \R^{m}$. Then $y'f(y)=0$ for $y\in\dualgamma$, and $y'f(y)>0$ for $y\notin\dualgamma$.
Choose
\begin{align*}
\pi(t)=f\Big(\psi(t, \alpha_t)b(t, \alpha_t)+\sigma(t, \alpha_t)\xi(t, \alpha_t)\Big),
\end{align*}
in \eqref{expectation}, then under \eqref{feasible}, the last integral is positive. So
for any $z\geq x \E e^{\int_0^Tr(t, \alpha_t)dt}$, there exists $\beta\geq 0$ such that $\E \big(X^\beta(T)\big)=z$.

Conversely, suppose that \eqref{optm} is feasible for any $z\geq x \E e^{\int_0^Tr(t, \alpha_t)dt}$. Then for any $z> x \E e^{\int_0^Tr(t, \alpha_t)dt}$, there is a $\pi\in\mathcal U$, such that $\E (X(T))=\E (X_0(T))+\E(X_1(T))=z$. Notice that $\E (X_0(T))=x \E e^{\int_0^Tr(t, \alpha_t)dt}$, thus it is necessary that there is a $\pi\in\mathcal U$ such that \[\E(X_1(T))=\E \int_0^T\pi(t)'(\psi(t, \alpha_t)b(t, \alpha_t)+\sigma(t, \alpha_t)\xi(t, \alpha_t))dt>0.\] If \eqref{feasible} was not true. Then $\psi(t, \alpha_t)b(t, \alpha_t)+\sigma(t, \alpha_t)\xi(t, \alpha_t)\in\dualgamma$ a.s, for a.e. $t\in[0, T]$. It would follow that
\[\pi(t)'(\psi(t, \alpha_t)b(t, \alpha_t)+\sigma(t, \alpha_t)\xi(t, \alpha_t))\leq 0, \ \mbox{a.s, for a.e. $t\in[0, T]$}, \]
for any $\pi\in\mathcal U$, leading to a contradiction.
\end{proof}

From the proof, we immediately see that
\begin{corollary}
The mean-variance problem \eqref{optm} is feasible for all $z\geq x \E e^{\int_0^Tr(t, \alpha_t)dt}$, if and only if, it is feasible for some $z>x \E e^{\int_0^Tr(t, \alpha_t)dt}$.
\end{corollary}

Because we are only interested in the feasible case, for the rest part of this subsection, we alway assume \eqref{feasible} holds.

\begin{remark}
From the proof, we also see that, when $\Gamma$ is symmetric, the mean-variance problem \eqref{optm} is feasible for all $z\in\R$, if and only if, it is feasible for some $z\neq x \E e^{\int_0^Tr(t, \alpha_t)dt}$.
\end{remark}
\subsection{Random regime switching market without portfolio constraint}
In this subsection, we assume the portfolio is unconstrained, i.e. $\Gamma=\mathbb{R}^m$.

In this case $\dualgamma=\{0\}$ and the feasible condition \eqref{feasible} is equivalent to
\[\E \int_0^T\Big|\psi(t, \alpha_t)b(t, \alpha_t)+\sigma(t, \alpha_t)\xi(t, \alpha_t)\Big|dt>0.\]
We remark that Problem \eqref{optm} is feasible for all $z\in\R$ under the above condition.
Furthermore, both \eqref{P1} and \eqref{P2} reduce to the same one multidimensional BSDE:
\begin{align}
\label{P}
\begin{cases}
dP(i)=-\Big[2r(i)P(i)-\frac{1}{P(i)}\left(P(i)b(i)+\sigma(i)\Lambda(i)\right)'(\sigma(i)\sigma(i)')^{-1}\left(P(i)b(i)+\sigma(i)\Lambda(i)\right)\\
\qquad\qquad\qquad+\sum\limits_{j=1}^{\ell} q_{ij}P(j)\Big]dt+\Lambda(i)'dW, \\
P(T, i)=1, \\
P(t, i)>0, \ \mbox{ for a.e. $t\in[0, T]$ and all $i\in\cM$.}
\end{cases}
\end{align}

From Theorem \ref{Riccatisingular}, we know \eqref{P} admits a unique solution $(P(i), \Lambda(i))_{i=1}^{\ell}$, such that $c\leq P(t, i)\leq C$ and $\int_0^\cdot\Lambda(s, i)dW(s)$ is a BMO martingale, for some constants $C>c>0$ and all $i\in\cM$.

To construct a solution for Problem \eqref{optmun}, we need to consider the following linear multidimensional BSDE:
\begin{align}
\label{H}
\begin{cases}
dH(i)=\Big[r(i)H(i)+b(i)'\left(\sigma(i)\sigma(i)'\right)^{-1}\sigma(i)\eta(i)\\
\qquad\qquad\quad+\frac{1}{P(i)}\Lambda(i)'\left(\sigma(i)'\left(\sigma(i)\sigma(i)'\right)^{-1}\sigma(i)-I_n\right)\eta(i)\\
\qquad\qquad\quad-\frac{1}{P(i)}\sum\limits_{j\neq i} q_{ij}P(j)(H(j)-H(i))\Big]dt+\eta(i)'dW, \\
H(T, i)=1, \mbox{ for all $i\in\cM$.}
\end{cases}
\end{align}
Its solution is defined as
\begin{definition}
A vector process $(H(i), \ \eta(i))_{i=1}^{\ell}$ is called a solution of the multidimensional BSDE \eqref{H}, if its satisfies \eqref{H}, and $(H(i), \ \eta(i))\in L^\infty_{\mathcal{F}^W}(0, T; \mathbb {R})\times \BMO$ for all $i\in\cM$.
\end{definition}

\begin{remark}
When $r$, $\mu$, $\sigma$ are deterministic, we have $H$ is deterministic and $\eta\equiv0$. Furthermore, when $m=n$, \eqref{P} and \eqref{H} coincide with the ODEs in \cite{ZY}.
\end{remark}

Briand and Confortola \cite{BC} obtained the existence and uniqueness of the solution of BSDEs with stochastic Lipschitz condition, but limited to $1$-dimensional case. The system \eqref{H} is a linear BSDE, but it does not satisfy the Lipschitz condition (because $\Lambda$ is unbounded). Furthermore, it is multidimensional, so their results can not be directly applied here.

We now address ourselves to the solvability of \eqref{H}. Define a closed convex set $\mathcal{B}$ as
\begin{align*}
\mathcal{B}=\Big\{U\in L^\infty_{\mathcal{F}^W}(0, T; \mathbb {R}^{\ell})\mid 0\leq e^{At}U(t, i)\leq B\ \mbox{ for a.e. $t\in[0, T]$ and every $i\in\cM$}\Big\}
\end{align*}
with the norm
\[|U|_\infty:=\max\limits_{i\in\cM}\underset{(t, \omega)\in[0, T]\times\Omega}{\esssup}e^{At}U(t, i), \]
where $A$ and $B$ are two positive scalars to be chosen later.
Then $(\mathcal{B}, |\cdot|_\infty)$ is a compact metric space.

We now use contraction mapping method to show
\begin{theorem}
\label{Hexis}
BSDE \eqref{H} admits a unique solution $(H(i), \ \eta(i))_{i=1}^{\ell}$.
\end{theorem}
\begin{proof}
We start with the existence.
Fixed any $U\in \mathcal{B}$. For each fixed $i\in\cM$, the following 1-dimensional linear BSDE, by \cite{Ko} (or \cite{CZ}),
\begin{align*}
\begin{cases}
dH(i)=\bigg[r(i)H(i)+\frac{H(i)}{P(i)}\sum\limits_{j\neq i} q_{ij}P(j)-\frac{1}{P(i)}\sum\limits_{j\neq i} q_{ij}P(j)U(j)+b(i)'\left(\sigma(i)\sigma(i)'\right)^{-1}\sigma(i)\eta(i)\\
\qquad\qquad\quad
+\frac{1}{P(i)}\Lambda(i)'\left(\sigma(i)'\left(\sigma(i)\sigma(i)'\right)^{-1}\sigma(i)-I_n\right)\eta(i) \bigg]dt+\eta(i)'dW, \\
H(T, i)=1,
\end{cases}
\end{align*}
has a unique adapted solution $(H(i), \ \eta(i))\in L^\infty_{\mathcal{F}^W}(0, T; \mathbb{R})\times \BMO$.
We call the map $U\mapsto (H(1), \ldots, H(\ell))$ as $\Theta$.

We next show $\Theta(\mathcal{B})\subset \mathcal{B}$ for proper chosen $A$ and $B$.
For each fixed $i\in\cM$, set
\[c(i)=\sigma(i)'\left(\sigma(i)\sigma(i)'\right)^{-1}b(i)+\frac{1}{P(i)}
\Big(\sigma(i)'\left(\sigma(i)\sigma(i)'\right)^{-1}\sigma(i)-I_n\Big)\Lambda(i).\]
By Theorem \ref{Riccatisingular}, $\int_0^\cdot c(i)'dW(s)$ is a BMO martingale and $\widetilde W^i(t):=W(t)+\int_0^tc(s, i)ds$ is a Brownian motion under the equivalent probability measure $\widetilde{\mathbb{P}}^i$ defined by
\[
\frac{d\widetilde{\mathbb{P}}^i}{d\mathbb{P}}\bigg|_{\mathcal{F}_T}=\mathcal{E}\bigg(-\int_0^Tc(s, i)'dW(s)\bigg).
\]
Let $\widetilde\E^i$ denote the corresponding expectation.
Let
\begin{align*}
e(t, i)=\exp\Bigg(-{\int_0^t\bigg[r(s, i)+\frac{1}{P(s, i)}\sum\limits_{j\neq i} q_{ij}P(s, j)\bigg]ds}\Bigg).
\end{align*}
Applying It\^{o}'s formula to $e(t, i)H(t, i)$, we have
\begin{align}
\label{Hrepre}
H(t, i)
=e(t, i)^{-1}\widetilde\E_t^i\bigg[e(T, i)+\int_t^T e(i) \frac{1}{P(i)}\sum\limits_{j\neq i} q_{ij}P(j)U(j)ds\bigg],
\end{align}
which is non-negative.
Because $e$, $e^{-1}$, $q$ and $P$ are all uniformly bounded,
it follows
\begin{align*}
e^{At}H(t, i)&\leq c \widetilde\E_t^i\bigg[e^{At}+\int_t^T e^{A(t-s)}\sum_{j\neq i}e^{As}U(j)ds\bigg] \\
&\leq c \widetilde\E_t^i\bigg[e^{AT}+\int_t^T e^{A(t-s)} |U|_\infty ds\bigg]\\
&\leq ce^{AT}+\frac{c}{A} |U|_\infty.
\end{align*}
This by definition leads to
\begin{align*}
|H|_\infty\leq ce^{AT}+\frac{c}{A} |U|_\infty.
\end{align*}
Note $c$ does not depend on $A$ in above.
It is not hard to see from the above inequality that $\Theta(\mathcal{B})\subset \mathcal{B}$, provided
\begin{align}
\label{B}
B=ce^{AT}+\frac{c}{A}B,
\end{align}
in which case $B=\frac{ce^{AT}}{1-c/A}$, so that $B$ goes to infinity if $A$ does so.

We now show $\Theta$ is a strict contraction, provided that $A$ is sufficiently large and $B$ satisfies \eqref{B}.
For any $U$, $\widetilde U\in\mathcal{B}$, let $H=\Theta(U)$, $\widetilde H=\Theta(\widetilde U)$, and set
\[\Delta H(t, i)=H(t, i)-\widetilde H(t, i), \ \text{and} \ \Delta U(t, i)=U(t, i)-\widetilde U(t, i).\]
Then by \eqref{Hrepre},
\begin{align*}
e^{At}|\Delta H(t, i)|&\leq e(t, i)^{-1}\widetilde\E_t^i\bigg[\int_t^T e^{A(t-s)} e(s, i) \frac{1}{P(s, i)}\sum\limits_{j\neq i} q_{ij}P(s, j)e^{As}|\Delta U(s, j)|ds\bigg],
\end{align*}
which implies, again by the boundedness of the coefficients,
\begin{align*}
e^{At}|\Delta H(t, i)| &\leq c|\Delta U|_\infty \widetilde\E_t^i\bigg[\int_t^T e^{A(t-s)}ds\bigg] \leq \frac{c}{A}|\Delta U|_\infty,
\end{align*}
or
\[ |\Delta H|_\infty\leq \frac{c}{A}|\Delta U|_\infty.\]
This means $\Theta$ is a strict contraction mapping on $\mathcal{B}$, provided $A>c$ and $B$ satisfies \eqref{B}.
Since $(\mathcal{B}, \ |\cdot|_\infty)$ is a compact metric space, the contraction mapping $\Theta$ has a fixed point $H$ in $\mathcal{B}$. Clearly, $(H, \eta)$ solves the system \eqref{H}. This proves the existence.

It is left to show the uniqueness. Suppose $(H, \eta)$ solves the system \eqref{H}.
If we can show that $H\geq 0$. Then because $H\in L^\infty_{\mathcal{F}^W}(0, T; \mathbb {R}^\ell)$, we have $H\in\mathcal{B}$ for $A$ sufficiently large and $B$ satisfying \eqref{B}.
Because $\Theta$ is a contraction mapping on $\mathcal{B}$, which has at most one fixed point, we conclude that \eqref{H} has at most one solution. Our problem now reduce to showing that $H\geq 0$.

The coefficients of \eqref{H} do not satisfy the Lipschitz condition in Lemma \ref{comparison}, so we cannot directly apply this lemma to prove that $H\geq 0$. But we can use the idea of its proof to deduce our conclusion. Applying It\^{o}'s formula to $(H(t, i)^{-})^{2}$, we have
\begin{align}
(H(t, i)^{-})^{2} =&-\int_{t}^{T}\bigg(2(H(i)^-)^{2}r(i)+\frac{2(H(i)^-)^{2}}{P(i)}\sum\limits_{j\neq i} q_{ij}P(j)\nonumber\\
&\qquad\qquad\quad+\frac{2H(i)^-}{P(i)}\sum\limits_{j\neq i} q_{ij}P(j)H(j)\bigg)ds\nonumber\\
& -\int_{t}^{T}2H(i)^-\eta(i)' d \widetilde W^i-\int_{t}^{T}I_{\{H(s, i)\leq 0\}}|\eta(i)|^{2}ds.
\label{Hgeq0}
\end{align}
By AM-GM inequality,
\[-H(i)^-H(j)=-H(i)^-H(j)^++H(i)^-H(j)^-\leq H(i)^-H(j)^-\leq \frac{1}{2}(H(i)^-)^{2}+\frac{1}{2}(H(j)^-)^{2}.\]
Dropping the last integral and using the the boundedness of the coefficients, we deduce from \eqref{Hgeq0} and the above inequality that
\begin{align*}
(H(t, i)^{-})^{2} &\leq c \int_{t}^{T}\Big((H(i)^-)^{2}+\sum_{j\neq i}(H(j)^-)^{2}\Big)ds-\int_{t}^{T}2H(i)^-\eta(i)' d \widetilde W^i(t).
\end{align*}
Taking conditional expectation $\widetilde\E_t^i$ on both sides gives
\begin{align*}
(H(t, i)^{-})^{2} &\leq c \int_{t}^{T} \sum_{j=1}^{\ell}\widetilde\E_t^i\big[(H(j)^-)^{2}\big] ds.
\end{align*}
Set
\[ E(t,i)=\underset{\omega\in\Omega}{\esssup}{\left(H(t,i)^-\right)^2}, \]
then
\[ E(t,i)\leq c \int_t^T \sum_{j=1}^{\ell}E(s,j)ds. \]
Thus
\[ 0\leq \sum_{j=1}^{\ell} E(t,j)\leq c\ell \int_t^T \sum_{j=1}^{\ell} E(s,j)ds.\]
It then follows from Gronwall's inequality that $\sum_{j=1}^{\ell} E(t,j)=0$, so
$H(t,i)\geq 0$ for a.e. $t\in[0,T]$ and all $i\in\cM$.
\end{proof}

\begin{remark}
If the interest rate $r(\cdot,\cdot,i)\geq0$, we can prove $H(t,i)\leq 1$, for a.e. $t\in[0,T]$ and all $i\in\cM$ by similar method as in Theorem \ref{Hexis}, which means that $H(t,i)$ is a genuine discount.
\end{remark}

Denote $K(i)=P(i)H(i)$, and $L(i)=P(i)\eta(i)+\frac{K(i)\Lambda(i)}{P(i)}$, then \eqref{H} is, by It\^{o}'s lemma, equivalent to
\begin{align*}
\begin{cases}
dK(i)=\bigg[\Big(b(i)'\left(\sigma(i)\sigma(i)'\right)^{-1}b(i)-r(i)
+\frac{\Lambda(i)'\sigma(i)'\left(\sigma(i)\sigma(i)'\right)^{-1}b(i)}{P(i)}\Big)K(i)\\
\qquad \qquad\quad+\big(b(i)+\frac{\sigma(i)\Lambda(i)}{P(i)}\big)'\left(\sigma(i)\sigma(i)'\right)^{-1}\sigma(i)L(i)
-\sum\limits_{j=1}^{\ell}q_{ij}K(j)\bigg]dt+L(i)'dW, \\
K(T, i)=1, \ \mbox{ for all $i\in\cM$.}
\end{cases}
\end{align*}
We use the process $K$ instead of $H$ to present our following results.

\begin{theorem}
\label{unconstraint}
The relaxed problem \eqref{optmun} has an optimal feedback control
\begin{align}
\label{portfolio}
\pi^*(t, X, i)&=-\left(\sigma(t, i)\sigma(t, i)'\right)^{-1}
\Bigg[\Big(b(t, i)+\frac{\sigma(t, i)\Lambda(t, i)}{P(t, i)}\Big)X\nonumber\\
&\qquad\qquad\qquad\qquad\qquad-(z+\lambda)
\frac{K(t, i)b(t, i)+\sigma(t, i)L(t, i)}{P(t, i)}\Bigg].
\end{align}
Moreover, the corresponding optimal value is
\begin{align}
\label{value}
\min_{\pi\in\mathcal{U}}\hat{J}(\pi, \lambda)
&=P(0, i_0)x^2-2(z+\lambda)K(0, i_0)x+(z+\lambda)^2-(z+\lambda)^2M-\lambda^2, 
\end{align}
where $M=\E\int_0^T O(t, \alpha_t) dt, $ and
\[O(t, i)=
\frac{\big(K(t, i)b(t, i)+\sigma(t, i)L(t, i)\big)'
\left(\sigma(t, i)\sigma(t, i)'\right)^{-1}
\big(K(t, i)b(t, i)+\sigma(t, i)L(t, i)\big)}
{P(t, i)}\]
for $i\in\cM$.
\end{theorem}
\begin{proof}
The proof is similar to that of Theorem \ref{LQcontrol}, so we leave the details to the interested readers. Just notice that by applying It\^{o}'s lemma to $P(t, \alpha_t)X(t)^2$ and $K(t, \alpha_t)X(t)$, we have
\begin{align*}
&\ \ \ \ \E(X(T)-(z+\lambda))^2\\
&=\E\big[P(T, \alpha_{T})X(T)^2
-2(z+\lambda)K(T, \alpha_{T})X(T)+(z+\lambda)^2\big]\\
&=P(0, i_0)x^2-2(z+\lambda)K(0, i_0)x+(z+\lambda)^2\\
&\quad+\E\int_0^{T}\Bigg\{P(s, \alpha_s)\big(\pi(s)-\pi^*(s, X(s), \alpha_s)\big)' \sigma(s, \alpha_s)\sigma(s, \alpha_s)'\big(\pi(s)-\pi^*(s, X(s), \alpha_s)\big)\\
&\qquad\qquad\qquad-(z+\lambda)^2O(s, \alpha_s)\Bigg\}ds.
\end{align*}
\end{proof}

\begin{theorem}
\label{efficientth}
The optimal portfolio of Problem \eqref{optm} corresponding to $\E (X(T))=z$, as a feedback function of the time $t$, the wealth level $X$, and the market regime $i$, is
\begin{align}
\label{efficient}
\pi^*(t, X, i)&=-\left(\sigma(t, i)\sigma(t, i)'\right)^{-1}
\Bigg[\Big(b(t, i)+\frac{\sigma(t, i)\Lambda(t, i)}{P(t, i)}\Big)X\nonumber\\
&\qquad\qquad\qquad\qquad\qquad-(z+\lambda^*)
\frac{K(t, i)b(t, i)+\sigma(t, i)L(t, i)}{P(t, i)}\Bigg], 
\end{align}
where
\begin{align*}
\lambda^*=\frac{z-Mz-K(0, i_0)x}{M}.
\end{align*}
The mean-variance frontier is
\begin{align}
\label{frontier}
\mathrm{Var} (X(T))=\frac{1-M}{M}\Big(\E (X(T))-\frac{K(0, i_0)}{1-M}x\Big)^2+\Big(P(0, i_0)-\frac{K(0, i_0)^2}{1-M}\Big)x^2,
\end{align}
with $0<M<1$.
\end{theorem}
\begin{proof}
Obviously, $O\geq 0$, so is $M$. If $M=0$, then $K(t, \alpha_t)b(t, \alpha_t)+\sigma(t, \alpha_t)L(t, \alpha_t)=0$.
Applying It\^{o}'s lemma to $K(t, \alpha_t)X(t)$, we have for any $\pi\in\mathcal{U}$,
\begin{align*}
\E(X(T))
&=K(0, i_0)x+\E\int_0^T\bigg\{\pi(t)'\big(K(t, \alpha_t)b(t, \alpha_t)+\sigma(t, \alpha_t)L(t, \alpha_t)\big)\\
&\qquad\qquad\qquad\qquad + X(t)\big(K(t, \alpha_t)b(t, \alpha_t)+\sigma(t, \alpha_t)L(t, \alpha_t)\big)'
\left(\sigma(t, i)\sigma(t, i)'\right)^{-1}\\
&\qquad\qquad\qquad\qquad \times\Big(b(t, \alpha_t)+\frac{\sigma(t, \alpha_t)\Lambda(t, \alpha_t)}{P(t, \alpha_t)}\Big)\bigg\}dt\\
&=K(0, i_0)x.
\end{align*}
This is a contradiction. Thus $M>0$. Therefore, by Theorem \ref{unconstraint},
\begin{align*}
\min_{\pi\in\mathcal{U}}\hat{J}(\pi, \lambda)
&=-M\lambda^2+2(z-Mz-K(0, i_0)x)\lambda+z^2-Mz^2+P(0, \alpha_0)x^2-2K(0, \alpha_0)zx
\end{align*}
is a strictly concave quadratic function of $\lambda$ so that
\begin{align*}
\max_{\lambda\in\mathbb{R}}\min_{\pi\in\mathcal{U}}\hat{J}(\pi, \lambda)
=\min_{\pi\in\mathcal{U}}\hat{J}(\pi, \lambda^{*}),
\end{align*}
where $\lambda^*$ is the unique maximizer
\begin{align*}
\lambda^*=\frac{z-Mz-K(0, i_0)x}{M}.
\end{align*}
This together with the duality relationship \eqref{duality}, by substituting $\lambda^*$ into \eqref{portfolio} and \eqref{value}, gives the optimal portfolio \eqref{efficient} and the optimal value
\begin{align*}
\mathrm{Var}(X(T))&=P(0, i_0)x^2-2(z+\lambda^*)K(0, i_0)x+(z+\lambda^*)^2-(z+\lambda^*)^2M-(\lambda^*)^2 \nonumber\\
&=\frac{1-M}{M}z^2-\frac{2xK(0, i_0)}{M}z+\frac{K(0, i_0)^2}{M}x^2+P(0, i_0)x^2.
\end{align*}
After completing square, this leads to the mean-variance frontier \eqref{frontier}, provided $M\neq 1$.
\par
We now show $M<1$ indeed. Write
\[\tilde{\Sigma}_{t}=I_n-\sigma(t, \alpha_t)'\big(\sigma(t, \alpha_t)\sigma(t, \alpha_t)'\big)^{-1}\sigma(t, \alpha_t)\]
which is is positive semidefinite by definition.
Applying It\^{o}'s formula to $P(t, \alpha_t)H(t, \alpha_t)^{2}$, we have
\begin{align*}
&\quad\;1-P(0, i_0)H(0, i_0)^{2}\\
&=\E \int_0^T\Bigg\{P(t, \alpha_t)H(t, \alpha_t)^{2}b(t, \alpha_t)'
\big(\sigma(t, \alpha_t)\sigma(t, \alpha_t)'\big)^{-1}b(t, \alpha_t)\\
&\qquad\qquad\quad+2H(t, \alpha_t) L(t, \alpha_t)'\sigma(t, \alpha_t)'
\big(\sigma(t, \alpha_t)\sigma(t, \alpha_t)'\big)^{-1}b(t, \alpha_t)
+\frac{L(t, \alpha_t)'L(t, \alpha_t)}{P(t, \alpha_t)}\\
&\qquad\qquad\quad+\frac{H(t, \alpha_t)^2}{P(t, \alpha_t)}\Lambda(t, \alpha_t)'
\tilde{\Sigma}_{t}\Lambda(t, \alpha_t)-\frac{2H(t, \alpha_t)}{P(t, \alpha_t)}L(t, \alpha_t)'
\tilde{\Sigma}_{t}\Lambda(t, \alpha_t)\\
& \qquad\qquad\quad+H(t, \alpha_t)^{2}\sum_{j=1}^{\ell} q_{\alpha_t j}P(t, j)
-2H(t, \alpha_t)\sum_{j=1}^{\ell} q_{\alpha_t j}P(t, j)H(t, j)\\
& \qquad\qquad\quad+\sum_{j=1}^{\ell} q_{\alpha_t j}P(t, \alpha_t)H(t, \alpha_t)^{2}
\Bigg\}dt\\
&=\E \int_0^T\Bigg\{P(t, \alpha_t)H(t, \alpha_t)^{2}b(t, \alpha_t)'
\big(\sigma(t, \alpha_t)\sigma(t, \alpha_t)'\big)^{-1}b(t, \alpha_t)\\
&\qquad\qquad\quad+2H(t, \alpha_t) L(t, \alpha_t)'\sigma(t, \alpha_t)'
\big(\sigma(t, \alpha_t)\sigma(t, \alpha_t)'\big)^{-1}b(t, \alpha_t)\\
& \qquad\qquad\quad+\frac{L(t, \alpha_t)'\sigma(t, \alpha_t)'\big(\sigma(t, \alpha_t)\sigma(t, \alpha_t)'\big)^{-1}
\sigma(t, \alpha_t)L(t, \alpha_t)}{P(t, \alpha_t)} \\
&\qquad\qquad\quad+\frac{L(t, \alpha_t)'\tilde{\Sigma}_{t}L(t, \alpha_t)}{P(t, \alpha_t)}+\frac{H(t, \alpha_t)^2}{P(t, \alpha_t)}\Lambda(t, \alpha_t)'
\tilde{\Sigma}_{t}\Lambda(t, \alpha_t)\\
&\qquad\qquad\quad-\frac{2H(t, \alpha_t)}{P(t, \alpha_t)}L(t, \alpha_t)'
\tilde{\Sigma}_{t}\Lambda(t, \alpha_t)+\sum_{j=1}^{\ell} q_{\alpha_tj}P(t, j)\Big(H(t, \alpha_t)-H(t, j)\Big)^2
\Bigg\}dt\\
&=\E \int_0^T\Bigg\{O(t, \alpha_t)+\frac{\Big(L(t, \alpha_t)-H(t, \alpha_t)\Lambda(t, \alpha_t)\Big)'\tilde{\Sigma}_{t}
\Big(L(t, \alpha_t)-H(t, \alpha_t)\Lambda(t, \alpha_t)\Big)}{P(t, \alpha_t)}\\
&\qquad\qquad\quad+\sum_{j\neq\alpha_{t}} q_{\alpha_tj}P(t, j)\Big(H(t, \alpha_t)-H(t, j)\Big)^2
\Bigg\}dt.
\end{align*}
Recall that $M=\E \int_0^T O(t, \alpha_t)dt$, $\tilde{\Sigma}_{t}\geq 0 $, and $q_{ij}\geq 0$ for any $i\neq j$, so the above gives $1-P(0, i_0)H(0, i_0)^2\geq M$, or $1-M\geq P(0, i_0)H(0, i_0)^2$.
On the other hand we have $P(0, i_0)>0$ by Theorem \ref{Riccatisingular}, and
\begin{align*}
H(0, i_0)
=\widetilde\E^i\left[e(T, i)+\int_0^T e(s, i) \frac{1}{P(s, i)}\sum\limits_{j\neq i} q_{ij}P(s, j)H(s, j)ds\right]> 0,
\end{align*}
by \eqref{Hrepre}. So we conclude that $M<1$.
\end{proof}

\begin{remark}
From the above proof, we see that the second term in \eqref{frontier} is always non-negative.
It becomes zero, only when the Markov chain $\alpha_t$ has only one state and $m=n$, namely we are in a complete market. This in theory confirms the assertion in Remark \ref{incompleteness}. Otherwise, it is positive, meaning that the systemic risk is positive (namely, one cannot perfectly hedge the risk).
\end{remark}

\begin{corollary}
The minimum variance point on the mean-variance frontier is
\begin{align*}
\bigg(\sqrt{P(0, i_0)-\frac{K(0, i_0)^2}{1-M}}x, \ \frac{K(0, i_0)}{1-M}x\bigg).
\end{align*}
Moreover, the corresponding optimal feedback portfolio is
\begin{align*}
\pi^*_{\mathrm{min}}(t, X, i)=-\left(\sigma(t, i)\sigma(t, i)'\right)^{-1}
\Bigg[\Big(&b(t, i)+\frac{\sigma(t, i)\Lambda(t, i)}{P(t, i)}\Big)\big(X-zH(t, i)\big)\\
&-z\sigma(t, i)\eta(t, i)\Bigg].
\end{align*}

\end{corollary}
\begin{remark}
Due to the minimum variance point, when the target $\E (X(T))$ is restricted to $\Big[\frac{K(0, i_0)}{1-M}x, \infty\Big)$ in \eqref{frontier}, one defines the efficient frontier for the mean-variance problem \eqref{optm}. And in this case, the
corresponding Lagrange multiplier $\lambda^*=\frac{z-Mz-K(0, i_0)}{M}\geq0$.
\end{remark}

\begin{theorem}[Mutual Fund Theorem]
Suppose an optimal portfolio $\pi^\star(\cdot)$ given by \eqref{efficient} corresponds to an expected return $z^{\star}>z_{\mathrm{min}}=\frac{K(0, i_0)}{1-M}x$. Then an admissible portfolio $\pi(\cdot)$ is efficient if and only if there exists a constant $\rho\geq0$ such that
\begin{align*}
\pi(t)=(1-\rho)\pi_{\mathrm{min}}^\star(t)+\rho\pi^\star(t), \ t\in[0, T].
\end{align*}
Moreover, the corresponding expected return is $(1-\rho)z_{\mathrm{min}}+\rho z^{\star}$.
\end{theorem}
The proof is similar to Theorem 5.3 in \cite{ZY}, we leave the details to the interested reader.
\subsection{Random regime switching market with no-shorting constraint}
Although our subsequent analysis can be applied to the case that not all the securities are allowed to short, i.e. $\Gamma=\mathbb{R}^{m_0}_+\times \mathbb{R}^{m-m_0}$ for some $m_0\leq m$. For notation simplicity, we simply consider the case that all the securities are not allowed to short, i.e. $\Gamma=\mathbb{R}^m_+$ in this subsection. In this case $\dualgamma=\mathbb{R}^m_-$.

In this subsection we assume
\begin{assumption} \label{assump4}
The interest process $r(\cdot)$ is deterministic, so that it is independent of the market regime process $\alpha$.
\end{assumption}
Under this assumption, $\psi(i)$ is a positive constant and $\xi(i)=0$ in Theorem \ref{theoremfeasible}, for every $i\in\cM$. So the feasible condition \eqref{feasible} is equivalent to
\begin{equation}
\label{feasiblecons}
\sum_{k=1}^m\E \int_0^Tb_k(t, \alpha_t)^+dt>0.
\end{equation}
Moreover, ESREs \eqref{P1} and \eqref{P2} become, respectively,
\begin{align}
\label{P11}
\begin{cases}
dP_1(i)=-\Big[(2rP_1(i)+H_1(P_1(i), \Lambda_1(i), i)+\sum\limits_{j=1}^{\ell}q_{ij}P_1(j)\Big]dt+\Lambda_1(i)'dW, \\
P_1(T, i)=1, \\
P_1(t, i)>0, \ \mbox{ for all $i\in\cM$;}
\end{cases}
\end{align}
and
\begin{align}
\label{P22}
\begin{cases}
dP_2(i)=-\Big[(2rP_2(i)+H_2(P_2(i), \Lambda_2(i), i)+\sum\limits_{j=1}^{\ell}q_{ij}P_2(j)\Big]dt+\Lambda_2(i)'dW, \\
P_2(T, i)=1, \\
P_2(t, i)>0, \ \mbox{ for all $i\in\cM$, }
\end{cases}
\end{align}
where
\begin{align*}
H_1(t, \omega, P, \Lambda, i)=\inf_{v\in\mathbb{R}_+^m}\big[v'P\sigma(t, i)\sigma(t, i)'v
+2v'(Pb(t, i)+\sigma(t, i)\Lambda)\big], \\
H_2(t, \omega, P, \Lambda, i)=\inf_{v\in\mathbb{R}_+^m}\big[v'P\sigma(t, i)\sigma(t, i)'v
-2v'(Pb(t, i)+\sigma(t, i)\Lambda)\big].
\end{align*}
Again by Theorem \ref{Riccatisingular}, we know \eqref{P11} and \eqref{P22} have solutions, which are denoted by $(P_1(i), \ \Lambda_1(i))_{i=1}^{\ell}$ and $(P_2(i), \ \Lambda_2(i))_{i=1}^{\ell}$, respectively, from now on.

\begin{lemma}
\label{Hcomparision}
Under Assumption \ref{assump4}, we have
\[
P_1(0, i_0)e^{-2\int_0^Tr(s)ds}\leq1, \ P_2(0, i_0)e^{-2\int_0^Tr(s)ds}<1.
\]
\end{lemma}
\begin{proof}
Consider the following BSDE with Lipschitz coefficients:
\begin{align*}
\begin{cases}
dP(i)=-\Big[(2rP(i)+\sum\limits_{j=1}^{\ell}q_{ij}P(j)\Big]dt+\Lambda(i)'dW, \\
P(T, i)=1, \ \mbox{ for all $i\in\cM$.}
\end{cases}
\end{align*}
It has a unique solution $\Big(e^{2\int_t^Tr(s)ds}, 0\Big)_{i=1}^{\ell}$. Notice that $H_1(t, P, \Lambda, i)\leq0$,
$H_2(t, P, \Lambda, i)\leq0$, we have $P_1(0, i_0)\leq e^{2\int_0^Tr(s)ds}$ and $P_2(0, i_0)\leq e^{2\int_0^Tr(s)ds}$ by Lemma \ref{comparison}.

Applying It\^{o}'s formula to $P_2(t, \alpha_t)e^{-2\int_t^Tr(s)ds}$, we get
\begin{align*}
1-P_2(0, i_0)e^{-2\int_0^Tr(s)ds}=-\E\int_0^Te^{-2\int_t^Tr(s)ds}H_2(t, P_2, \Lambda_2, \alpha_t)dt.
\end{align*}
Now suppose $P_2(0, i_0)=e^{2\int_0^Tr(s)ds}$. Then $H_2(t, P_2, \Lambda_2, \alpha_t)=0$ and $P_2(t, \alpha_t)=e^{2\int_t^Tr(s)ds}$ for $t\in[0, T]$.
Thus $\Big(e^{2\int_t^Tr(s)ds}, 0\Big)_{i=1}^{\ell}$ is the unique solution of \eqref{P22}.
Consequently, $H_2(t, P_2, 0, \alpha_t)=0$ for $t\in[0, T]$. It follows
\begin{align*}
0=H_2(t, P_2, 0, \alpha_t)&=P_2\inf_{v\in\mathbb{R}_+^m}\big[v'\sigma(t, \alpha_t)\sigma(t, \alpha_t)'v
-2v'b(t, \alpha_t)\big] \leq P_2\inf_{v\in\mathbb{R}_+^m}\big[Cv'v
-2v'b(t, \alpha_t)\big]
\end{align*}
where $C>0$. By choosing $v_{t}=\varepsilon(b_1(t, \alpha_t)^+, \ldots, b_m(t, \alpha_t)^+)\in\mathbb{R}_+^m $ with $\varepsilon>0$ in above, we get
\begin{align*}
0=\E\int_{0}^{T}H_2(t, P_2, 0, \alpha_t)
&\leq (C\varepsilon^{2}-2\varepsilon)\E\int_{0}^{T} e^{2\int_t^Tr(s)ds}\sum_{k=1}^{m}(b_k(t, \alpha_t)^+)^{2}dt.
\end{align*}
Noticing \eqref{feasiblecons}, we see the right hand side is negative for sufficiently small $\varepsilon>0$, leading to a contraction.
Therefore $P_2(0, i_0)e^{-2\int_0^Tr(s)ds}<1$.
\end{proof}

In the present setting, we have
\begin{align*}
\hat v_1(t, \omega, P, \Lambda, i)=\argmin_{v\in\mathbb{R}_+^m}\big[v'(P\sigma(t, i)\sigma(t, i)')v
+2v'(Pb(t, i)+\sigma(t, i)\Lambda)\big], \\
\hat v_2(t, \omega, P, \Lambda, i)=\argmin_{v\in\mathbb{R}_+^m}\big[v'(P\sigma(t, i)\sigma(t, i)')v
-2v'(Pb(t, i)+\sigma(t, i)\Lambda)\big].
\end{align*}

As for the relaxed problem \eqref{optmun}, we have the following analog result of Theorem \ref{LQcontrol}.
\begin{theorem}
Under Assumption \ref{assump4}, the relaxed problem \eqref{optmun} has an optimal feedback control
\begin{align*}
\pi^*(t, X, i)&=\hat v_1(t, P_1, \Lambda_1, i)\bigg(X-(\lambda+z)e^{-\int_t^Tr(s)ds}\bigg)^+\nonumber\\
&\qquad+\hat v_2(t, P_2, \Lambda_2, i)\bigg(X-(\lambda+z)e^{-\int_t^Tr(s)ds}\bigg)^-.
\end{align*}
Moreover, the corresponding optimal value is
\begin{align*}
\min_{\pi\in\mathcal{U}}\hat{J}(\pi, \lambda)
=P_1(0, i_0)(x-(\lambda+z)e^{-\int_0^Tr(s)ds})_+^2+P_2(0, i_0)(x-(\lambda+z)e^{-\int_0^Tr(s)ds})_-^2
-\lambda^2.
\end{align*}
\end{theorem}

Next we will find the best Lagrange multiplier $\lambda^*$. Clearly
\begin{align*}
\min_{\pi\in\mathcal{U}}\hat{J}(\pi, \lambda)
&=\begin{cases}
f(\lambda), &\quad \mbox{if} \ \lambda\leq xe^{\int_0^Tr(s)ds}-z;\\
g(\lambda), &\quad \mbox{if} \ \lambda\geq xe^{\int_0^Tr(s)ds}-z,
\end{cases}
\end{align*}
where
\begin{align*}
f(\lambda)&=(P_1(0, i_0)e^{-2\int_0^Tr(s)ds}-1)\lambda^2+2 P_1(0, i_0)e^{-\int_0^Tr(s)ds}(ze^{-\int_0^Tr(s)ds}-x)\lambda\\
&\qquad+P_1(0, i_0)(x-ze^{-\int_0^Tr(s)ds})^2, \\
h(\lambda)&=(P_2(0, i_0)e^{-2\int_0^Tr(s)ds}-1)\lambda^2+2 P_2(0, i_0)e^{-\int_0^Tr(s)ds}(ze^{-\int_0^Tr(s)ds}-x)\lambda\\
&\qquad+P_2(0, i_0)(x-ze^{-\int_0^Tr(s)ds})^2.
\end{align*}
Using $z\geq xe^{\int_0^Tr(s)ds}$, $P_1(0, i_0)e^{-2\int_0^Tr(s)ds}\leq1$ and $P_2(0, i_0)e^{-2\int_0^Tr(s)ds}<1$ by Lemma \ref{Hcomparision}, one can easily deduce
\begin{align*}
\max_{\lambda\leq xe^{\int_0^Tr(s)ds}-z}f(\lambda)&=f(xe^{\int_0^Tr(s)ds}-z), \\
\max_{\lambda\geq xe^{\int_0^Tr(s)ds}-z}h(\lambda)&=h(\lambda^*)
=\frac{P_2(0, i_0)e^{-2\int_0^Tr(s)ds}}{1-P_2(0, i_0)e^{-2\int_0^Tr(s)ds}}\Big(z-xe^{\int_0^Tr(s)ds}\Big)^2,
\end{align*}
where
\[\lambda^*=\frac{P_2(0, i_0)e^{-\int_0^Tr(s)ds}(ze^{-\int_0^Tr(s)ds}-x)}{1-P_2(0, i_0)e^{-2\int_0^Tr(s)ds}}
\geq xe^{\int_0^Tr(s)ds}-z.\]
Furthermore,
\[h(\lambda^*)\geq h(xe^{\int_0^Tr(s)ds}-z)=f(xe^{\int_0^Tr(s)ds}-z).\]
Thus $\lambda^*$ attains the maximum of $\min\limits_{\pi\in\mathcal{U}}\hat{J}(\pi, \lambda)$.

The above analysis boils down to the following theorem.
\begin{theorem}
Suppose Assumption \ref{assump4} holds. The optimal portfolio of Problem \eqref{optm} corresponding to $\E (X(T))=z$, as a feedback function of the time $t$, the wealth level $X$, and the market regime $i$, is
\begin{align*}
\pi^*(t, X, i)=-\hat v_2(t, P_2, \Lambda_2, i)\Big(X-(\lambda^*+z)e^{-\int_t^Tr(s)ds}\Big).
\end{align*}
The efficient frontier is
\begin{align*}
\mathrm{Var}(X(T))=\frac{P_2(0, i_0)e^{-2\int_0^Tr(s)ds}}{1-P_2(0, i_0)e^{-2\int_0^Tr(s)ds}}\Big(\E(X(T))-xe^{\int_0^Tr(s)ds}\Big)^2,
\end{align*}
where $\E(X(T))\geq xe^{\int_0^Tr(s)ds}.$
\end{theorem}

\begin{remark}
In this case, we have assumed that the interest rate $r$ is a deterministic function which is independent of $\omega$ and the Markov chain, thus the risk adjust process $H(t, i)$ of $\eqref{H}$ must be of the form $(H(t, i), \ \eta(t, i))=(e^{-\int_t^Tr(s)ds}, \ 0)$ for all $i\in\cM$. Then from the proof of Theorem \ref{efficientth}, we know $1-\frac{K(0, i_0)^2}{P(0, i_0)}=M$, which leads the efficient frontier above a harf line, even though the number of the stock may less than the dimension of the Brownian motion and the appearance of the Markov chain. Economically speaking, one can put all the money into the risk-free asset to reduce the risk to 0.
\end{remark}
\section{Concluding remarks}
In this paper, we developed a constrained stochastic LQ problem with regime switching and random coefficients. And we succeeded in obtaining the optimal state feedback control and optimal cost value via two systems of highly nonlinear BSDEs which are introduced in this paper for the first time. The solvability of these two systems of equations is interesting in its own from the point view of BSDE theory. At last, we solved two continuous-time mean-variance portfolio selection problems with regime switching and random coefficients with/without trading constraint by a system of linear BSDEs with unbounded coefficients.
 Extensions in other directions can be interesting as well. For instance, (1) The mean-variance portfolio selection problem with no-shorting constraint if the interest rate $r$ is a stochastic process. (2) The constrained LQ control problem with regime switching in infinite time horizon with deterministic or random coefficients. (3) The solvability of matrix-valued system of ESREs.

\begin{appendix}
\section*{Proof of Lemma 3.4}\label{appn} 
For $t\in[0, T]$ and every $i\in\cM$, set
\[\delta Y(t, i)=Y(t, i)-\overline Y(t, i), \ \delta Z(t, i)=Z(t, i)-\overline Z(t, i).\]
Applying It\^{o}'s formula to $(\delta Y(t, i)^+)^2$, we have
\begin{align*}
\mathbb E(\delta Y(t, i)^+)^2&=\mathbb E\int_t^T 2\delta Y(s, i)^+\Big[f(s, Y(s, i), Y(s, -i), Z(s, i), i)\\
&\qquad\qquad\qquad\qquad\qquad-\overline f(s, \overline Y(s, i), \overline Y(s, -i), \overline Z(s, i), i)\Big]ds\\
& \qquad-\mathbb E\int_t^T I_{\delta Y(s, i)\geq 0}|\delta Z(s, i)|^2ds\\
&\leq \mathbb{E}\int_t^T 2c\delta Y(s, i)^+(|\delta Y(s, i)|+\sum_{j\neq i}\delta Y(s, j)^++\delta Z(s, i))ds\\
& \qquad-\mathbb E\int_t^T I_{\delta Y(s, i)\geq 0}|\delta Z(s, i)|^2ds\\
&\leq c\mathbb{E}\int_t^T \sum_{i=1}^{\ell}(\delta Y(s, i)^+)^2ds,
\end{align*}
by the AM-GM inequality. Thus
\[\sum_{i=1}^{\ell}\mathbb E(\delta Y(t, i)^+)^2 \leq c\int_t^T \sum_{i=1}^{\ell}\mathbb{E}(\delta Y(s, i)^+)^2ds.\]
It then follows from Gronwall's inequality that $\sum_{i=1}^{\ell}\mathbb E(\delta Y(t, i)^+)^2=0$, thus $Y(t, i)\leq\overline Y(t, i)$ for a.e. $t\in[0, T]$ and all $i\in\cM$.
\end{appendix}

\section*{Acknowledgements}
The authors wish to thank the anonymous referee and editors for their insightful and constructive comments and suggestions on the previous version of this paper.

The first author is partially  supported by Lebesgue
Center of Mathematics \textquotedblleft Investissements d'avenir\textquotedblright program-ANR-11-LABX-0020-01,  ANR CAESARS (No. 15-CE05-0024) and  ANR MFG (No. 16-CE40-0015-01).

The second author is partially supported by NSFC (No.~11801315 and No.~71871129), NSF of Shandong Province (No.~ZR2018QA001 and No.~ZR2020MA032), and the Colleges and Universities Youth Innovation Technology Program of Shandong Province (No.~2019KJI011).

The third author is partially supported by NSFC (No.~11971409), Hong Kong
GRF (No.~15204216 and No.~15202817), The PolyU-SDU Joint Research Center on Financial Mathematics and the CAS AMSS-PolyU Joint Laboratory of Applied Mathematics, The Hong Kong Polytechnic University.


\begin{thebibliography}{4}
\bibitem {Bi}
\textsc{Bismut M. } (1976).
Linear quadratic optimal stochastic control with random coefficients.
\textit{SIAM J. Control Optim.}, 14(3):419-444.


\bibitem {BC}
\textsc{Briand P. and Confortola F.} (2008).
BSDEs with stochastic Lipschitz condition and quadratic PDEs in Hilbert spaces.
\textit{Stochastic Process. Appl.}, 118(5):818-838.

\bibitem {CLZ}
\textsc{Chen S., Li X., and Zhou X.} (1998).
Stochastic linear quadratic regulators with indefinite control weight costs.
\textit{SIAM J. Control Optim.}, 36(5):1685-1702.


\bibitem {CZ}
\textsc{Cvitanic J. and Zhang J.} (2012).
\textit{Contract theory in continuous-time models}.
Springer Science and Business Media.

\bibitem {CS}
\textsc{Czichowsky C. and Schweizer M.} (2013).
Cone-constrained continuous-time Markowitz problems.
\textit{Ann. Appl. Probab.}, 23(2):764-810.

\bibitem {Du}
\textsc{Du K.}) (2015).
Solvability conditions for indefinite linear quadratic optimal stochastic control problems and associated stochastic Riccati equations.
\textit{SIAM J. Control Optim.}, 53(6):3673-3689.

\bibitem {Ga}
\textsc{Gal'chuk L I.} (1979).
Existence and uniqueness of a solution for stochastic equations with respect to semimartingales.
\textit{Theory Probab. Appl.}, 23(4):751-763.

\bibitem {HLT}
\textsc{Hu Y., Liang G., and Tang S.} (2020).
Systems of infinite horizon and ergodic BSDE arising in regime switching forward performance processes.
\textit{SIAM J. Control optim.}, 58(4):2503-2534.

\bibitem {HP}
\textsc{Hu Y. and Peng S.} (2005).
On the comparison theorem for multidimensional BSDEs.
\textit{C. R. Acad. Sci. Paris, Ser. I}, 343(2):135-140.

\bibitem {HT}
\textsc{Hu Y. and Tang S.} (2016) .
Multi-dimensional backward stochastic differential equations of diagonally quadratic generators.
\textit{Stochastic Process. Appl.}, 126(4):1066-1086.

\bibitem {HZ03}
\textsc{Hu Y. and Zhou X.} (2003).
Indefinite stochastic Riccati equations.
\textit{SIAM J. Control Optim.}, 42(1):123-137.

\bibitem {HZ}
\textsc{Hu Y. and Zhou X.} (2005).
Constrained stochastic LQ control with random coefficients, and application to portfolio selection.
\textit{SIAM J. Control Optim.}, 44(2):444-466.

\bibitem {Ka}
\textsc{Kazamaki N.} (2016).
\textit{Continuous exponential martingales and BMO.}
Springer.

\bibitem {Ko}
\textsc{Kobylanski M.} (2000).
Backward stochastic differential equations and partial differential equations with quadratic growth.
\textit{Ann. Probab.}, 28(2):558-602.

\bibitem {KT}
\textsc{Kohlmann M. and Tang S.} (2002).
Global adapted solution of one-dimensional backward stochastic Riccati equations, with application to the mean-variance hedging.
\textit{Stochastic Process. Appl.}, 97(2):255-288.

\bibitem {KZ}
\textsc{Kohlmann M. and Zhou X.} (2000).
Relationship between backward stochastic differential equations and stochastic controls: a linear-quadratic approach.
\textit{SIAM J. Control Optim.}, 38(5):1392-1407.

\bibitem {LM}
\textsc{Lepeltier J. and Mart\'{\i}n J.} (1998).
Existence for BSDE with
superlinear-quadratic coefficient.
\textit{Stochastics Stochastics Rep.}, 63(3-4):227-240.

\bibitem {LN}
\textsc{Li D. and Ng W.} (2000).
Optimal Dynamic Portfolio Selection:
Multiperiod Mean-Variance Formulation.
\textit{Math. Finance.}, 10(3):387-406.

\bibitem {LxZ}
\textsc{Li X. and Zhou X.} (2002).
Indefinite stochastic LQ controls with Markovian jumps in a finite time horizon.
\textit{Commun. Inf. Syst.}, 2(3):265-282.

\bibitem {LZR}
\textsc{Li X., Zhou X., and Rami M.} (2003).
Indefinite stochastic linear quadratic control with Markovian jumps in infinite time horizon. \textit{J. Global Optim.}, 27(2-3):149-175.


\bibitem {LZL}
\textsc{Li X., Zhou X., and Lim A.} (2002).
Dynamic mean-variance portfolio selection with no-shorting constraints.
\textit{SIAM J. Control Optim.}, 40(5):1540-1555.

\bibitem {LZ}
\textsc{Lim A. and Zhou X.} (2002)
Mean-variance portfolio selection with random parameters in a complete market.
\textit{Math. Oper. Res.},27(1):101-120.

\bibitem {Lu}
\textsc{Luenberger D.} (1997)
\textit{Optimization by vector space methods.} John
Wiley and Sons.

\bibitem {QZ}
\textsc{Qian Z. and Zhou X.} (2013).
Existence of solutions to a class of indefinite stochastic Riccati equations.
\textit{SIAM J. Control Optim.}, 51(1):221-229.

\bibitem {Ta}
\textsc{Tang S.} (2003).
General linear quadratic optimal stochastic control problems with random coefficients: linear stochastic Hamilton systems and backward stochastic Riccati equations.
\textit{SIAM J. Control Optim.}, 42(1):53-75.

\bibitem {Wo}
\textsc{Wonham W.} (1968).
On a matrix Riccati equation of stochastic control.
\textit{SIAM J. Control.}, 6(4):681-697.

\bibitem {YZ}
\textsc{Yong J. and Zhou X.} (1999).
\textit{Stochastic controls: Hamiltonian systems and HJB equations.} Springer-Verlag. New York.

\bibitem {Yu}
\textsc{Yu Z.} (2013).
Continuous-time mean-variance portfolio selection with random horizon.
\textit{Appl. Math. Optim.}, 68(3):333-359.

\bibitem {ZL}
\textsc{Zhou X. and Li D.} (2000).
Continuous-time mean-variance portfolio selection: A stochastic LQ framework.
\textit{Appl. Math. Optim.}, 42(1):19-33.

\bibitem {ZY}
\textsc{Zhou X. and Yin G.} (2003).
Markowitz's mean-variance portfolio selection with regime switching: A continuous-time model.
\textit{SIAM J. Control Optim.}, 42(4):1466-1482.
\end{thebibliography}
\end{document}